\documentclass[10pt,a4paper,reqno]{amsart}%
\usepackage{amssymb}
\usepackage{amsmath}
\usepackage{amsfonts}
\usepackage[T1]{fontenc}
\usepackage{color}
\usepackage[usenames,dvipsnames]{xcolor}
\usepackage{bbm}
\usepackage{graphicx}
\usepackage{rotating}
\usepackage{mathpazo}
\usepackage{hyperref}
\usepackage{mathrsfs}
\usepackage[all]{xy}%
\providecommand{\U}[1]{\protect\rule{.1in}{.1in}}
\newtheorem{theorem}{Theorem}

\newtheorem{corollary}[theorem]{Corollary}

\newtheorem{definition}[theorem]{Definition}
\newtheorem{example}[theorem]{Example}

\newtheorem{proposition}[theorem]{Proposition}
\newtheorem{remark}[theorem]{Remark}

\thanks{}
\email{lvitagliano@unisa.it}
\begin{document}
\title[Representations of SH Lie-Rinehart ALgebras]{{Representations of Homotopy Lie-Rinehart Algebras}}
\author{Luca Vitagliano}
\address{DipMat, Universit\`a degli Studi di Salerno, {\& Istituto Nazionale di Fisica
Nucleare, GC Salerno,} Via Ponte don Melillo, 84084 Fisciano (SA), Italy.}

\begin{abstract}
I propose a definition of left/right connection along a strong homotopy
Lie-Rinehart algebra. This allows me to generalize simultaneously
representations up to homotopy of Lie algebroids and actions of $L_{\infty} $
algebras on graded manifolds. I also discuss the Schouten-Nijenhuis calculus
associated to strong homotopy Lie-Rinehart connections.

\end{abstract}
\maketitle


\quad

\emph{MSC-class}: 16W25, 53B05, 58A50, 14F10.

\tableofcontents

\section*{Introduction}

Let $\mathfrak{g}$ be a vector space over a field $K$. Lie brackets in
$\mathfrak{g}$ correspond bijectively to DG coalgebra structures on the
exterior coalgebra $\Lambda_{K}^{c}\mathfrak{g}$ (and to Gerstenhaber algebra
structures on the exterior algebra $\Lambda_{K}\mathfrak{g}$). Moreover, the
homology of $\Lambda_{K}^{c}\mathfrak{g}$ is that of $\mathfrak{g}$. Now, let
$L$ be a module over an associative, commutative, unital algebra $A$. In
\cite{h98}, Huebschmann remarks that there is no way to endow $\Lambda_{A}%
^{c}L$ with a DG coalgebra structure corresponding to a given Lie-Rinehart
(LR) structure on $(A,L)$ (see Section \ref{SecLRA} for a remainder on the
notion of LR algebra). Instead, LR algebra structures on $(A,L)$ correspond
bijectively to Gerstenhaber algebra stuctures on the exterior algebra
$\Lambda_{A}L$. What is then the relation between Gerstenhaber structures and
the cohomology of LR algebras? Huebschmann finds an answer in terms of
Batalin-Vilkovisky (BV) algebras (I refer to \cite{h98} for the notion of
Gerstenhaber, BV algebras and their relation with LR algebras). In particular,
he algebrizes a result of Koszul \cite{k85} (see also \cite{x99}) and shows
that, given an LR-algebra $(A,L)$,

\begin{enumerate}
\item BV algebra structures on $\Lambda_{A}L$ correspond bijectively to right
$(A,L)$-module structures in $A$, and

\item if $L$ is projective as an $A$-module, then a BV algebra structure in
$\Lambda_{A}L$ computes the homologies of $(A,L)$ with coefficients in the
right module $A$.
\end{enumerate}

{In \cite{kk12} Kowalzig and Kr\"{a}hmer describe a rather general cohomology
theory which is able to account for the algebraic structure of many different
cohomology theories (Hochschild, LR, group cohomology, etc.).
Kowalzig and Kr\"{a}hmer's theory encompasses Huebschmann's results.} On
another hand, Huebschmann himself explored in \cite{h05} higher homotopy
generalizations of LR, Gerstenhaber, and BV algebras with the aim of
\guillemotleft unify[ing] these structures by means of the relationship
between Lie-Rinehart, Gerstenhaber, and Batalin-Vilkovisky
algebras\guillemotright \ first observed in \cite{h98}, and the hope that this
will have been \guillemotleft a first step towards taming the bracket zoo that
arose recently in topological field theory\guillemotright . The higher
homotopies which are exploited in \cite{h05} \guillemotleft are of a special
kind, though, where only the first of an (in general) infinite family is non-zero\guillemotright .

In the literature, there already exist higher homotopy generalizations of LR
\cite{k01}, Gerstenhaber \cite{b11}, and BV \cite{k00} algebras (see also
\cite{v99}, and \cite{gtv12} for an operadic approach to homotopy Gerstenhaber
and homotopy BV algebras respectively). One of the aims of this paper is to
generalize Huebschmann's results \cite{h98,h05} to a setting where all higher
homotopies (in the infinite family) are possibly non-zero. To achieve this
goal, I generalize first the notions of (left/right) LR connection, and
(left/right) LR module \cite{h98}. As a byproduct, I obtain a rather wide
generalization of various constructions scattered in the literature. Namely,
the $LR_{\infty}$ modules defined in this paper generalize

\begin{enumerate}
\item \emph{representations up to homotopy} of Lie algebroids \cite{ac11}.
Namely, $LR_{\infty}$ modules may be understood as representations of
$LR_{\infty}$ algebras, which are homotopy versions of LR algebras, which, in
their turn, are purely algebraic generalizations of Lie algebroids.

\item \emph{actions of }$L_{\infty}$\emph{ algebras on graded manifolds
}\cite{mz12}. Namely, $LR_{\infty}$ algebras generalize $L_{\infty}$ algebras
and actions of the latter on graded manifolds are special instances of actions
of $LR_{\infty}$ algebras on graded algebra extensions (see Section
\ref{SecAct}).

\item \emph{actions of Lie algebroids on fibered manifolds} and derivative
representations of Lie algebroids\emph{ }\cite{k-sm02}. Namely, similarly as
above, actions of $LR_{\infty}$ algebras on graded algebra extensions are
purely algebraic and homotopy versions of actions of Lie algebroids on fibered manifolds.
\end{enumerate}

Finally, I obtain a generalization, to the homotopy setting, of the standard
Schouten-Nijenhuis calculus on multivectors and, more generally, exterior
algebras of Lie algebroids.

On another hand (left/right) modules over LR algebras are key concepts in the
theory of $\mathcal{D}$-modules \cite{p83,mp89}. Recall that a $\mathcal{D}%
$-module is a (left/right) module over the algebra $\mathcal{D}$ of linear
differential operators on a manifold. Since $\mathcal{D}$ is the universal
enveloping algebra of the LR algebra of vector fields, a left $\mathcal{D}%
$-module is actually the same as a module with a flat connection, i.e., a
module with a left representation of the LR algebra of vector fields. This
explains the relationship with LR algebras. $\mathcal{D}$-modules provide a
natural language for a geometric theory of linear partial differential
equations (PDE), and define a rich homological algebra \cite{s94}. The datum
of a linear PDE can be encoded by a $\mathcal{D}$-module, whose homological
algebra contains relevant information about the PDE (symmetries, conservation
laws, etc.). More generally, the datum of a non-linear PDE can be encoded by a
diffiety (or a $\mathcal{D}$-scheme) i.e., a countable-dimensional manifold
with a finite-dimensional, involutive distribution. Vector fields in the
distribution form again an LR algebra, and, similarly as before, modules over
this LR algebra contain relevant information about the non-linear PDE. The
idea of building a theory of $\mathcal{D}$-modules (and $\mathcal{D}$-schemes)
up to homotopy is intriguing. This paper may represent a first (short) step in
this direction.

The paper is divided into six sections and three appendixes. In Section
\ref{SecLR} I recall the definitions of LR algebras, (flat) connections along
them and the associated (co)homologies. In Section \ref{SecDer} I set the
notations about (multi)derivations of graded algebras and modules, and discuss
Lie algebras of multiderivations. In Section \ref{SecSHLR} I define strong
homotopy (SH) LR algebras in terms of multiderivations. It is a simple
translation exercise to check that Definition \ref{DefSHLR} is actually
equivalent to Kjeseth's definition of \emph{homotopy Lie-Rinehart pairs}
\cite{k01}. A SH LR algebra is a homotopy versions of a LR algebra. Notice,
however, that it is not the most general possible one. Roughly speaking, while
in a SH LR algebra the Lie bracket, and the Lie module structures (of a LR
algebra) are \textquotedblleft\emph{homotopyfied}\textquotedblright, the
commutative algebra structure is not. Nonetheless, SH LR algebras are already
enough for many applications. Namely, they appear in the BRST-BV formalism
\cite{k01b}, foliation theory \cite{h05,v12,v12b}, complex geometry
\cite{y11}. Moreover, as shown in Section \ref{SecSHLR} (see also \cite{b11}),
one can associate a SH LR algebra to a homotopy Poisson algebra (in the sense
of Cattaneo-Felder \cite{cf07}). In their turn homotopy Poisson algebras
appear in Poisson geometry \cite{cf07} (see also \cite{op05}). 
All these geometric examples provide
motivations for the abstract content of this manuscript, and the application
of the general theory to them will be discussed elsewhere. In Section
\ref{SecLeftConn} I define left connections along SH LR algebras. To my
knowledge, this definition appears here for the first time. I also discuss
cohomologies associated to flat connections and generalize to the present
homotopy setting actions of LR algebras and Schouten-Nijenhuis calculus. As
noted above, actions of SH LR algebras unify two constructions separately
present in literature: namely representations up to homotopy of Lie algebroids
\cite{ac11}, and actions of $L_{\infty}$ algebras on graded manifolds
\cite{mz12}. In Section \ref{SecRightConn} I define right connections along
SH\ LR algebras, cohomologies of flat right connections and the right homotopy
Schouten-Nijenhuis calculus. I also present the homotopy version of the
interplay between LR algebras and BV algebras \cite{h98}. A further motivation
for studying both left and right connections along SH LR algebras is that the
SH LR algebra of a homotopy Poisson algebra comes equipped with a canonical
flat right connection. Moreover, it can be showed that the $A_{\infty}$
algebra of a foliation $\mathcal{F}$ carries both a flat left and a flat right
connection along the SH LR algebra of $\mathcal{F}$ \cite{v12b}. In Section
\ref{SecDerRep} I generalize to the homotopy (and purely algebraic) setting
the concept of derivative representation of a Lie algebroid \cite{k-sm02}.
Finally, Appendixes \ref{ApForm} and \ref{ApOpRight} contain complementary
material, while Appendix \ref{Tables} contains three tables illustrating how
the constructions in this paper correspond to, generalize, and unify
constructions already present in literature.

\subsection*{Conventions and notations}

I will adopt the following notations and conventions throughout the paper. Let
$\ell,m$ be positive integers. I denote by $S_{\ell,m}$ the set of $(\ell
,m)$\emph{-unshuffles}, i.e., permutations $\sigma$ of $\{1,\ldots,\ell+m\}$
such that
\[
\sigma(1)<\cdots<\sigma(\ell),\quad\text{and\quad}\sigma(\ell+1)<\cdots
<\sigma(\ell+m).
\]

If $S$ is a set, I denote
\[
S^{\times k}:={}\underset{k\text{ times}}{\underbrace{S\times\cdots\times S}%
}.
\]

Every vector space will be over a field $K$ of zero characteristic. The degree
of a homogeneous element $v$ in a graded vector space will be denoted by
$\bar{v}$. However, when it appears in the exponent of a sign $(-)$, I will
always omit the overbar, and write, for instance, $(-)^{v}$ instead of
$(-)^{\bar{v}}$.

Let $V$ be a graded vector space,
\[
\boldsymbol{v}=(v_{1},\ldots,v_{n})\in V^{\times n},
\]
and $\sigma$ a permutation of $\{1,\ldots,n\}$. I denote by $\alpha
(\sigma,\boldsymbol{v})$ the sign implicitly defined by
\[
v_{\sigma(1)}\odot\cdots\odot v_{\sigma(n)}=\alpha(\sigma,\boldsymbol{v}%
)\,v_{1}\odot\cdots\odot v_{n}%
\]
where $\odot$ is the graded symmetric product in the graded symmetric algebra of $V$.

In this paper, I will deal with algebraic structures generalizing $L_{\infty}$
algebras and their modules (see, for instance, \cite{ls93,lm95}). $L_{\infty}$
algebras, also named \emph{strong homotopy (SH) Lie algebras}, are homotopy
versions of Lie algebras, i.e., Lie algebras \emph{up to homotopy}. More
precisely, an $L_{\infty}$ \emph{algebra} is a graded vector space $V$
equipped with a family of $k$-ary, graded skew-symmetric, multilinear, degree
$2-k$ operations
\[
\lambda_{k}:{}V^{\times k}\longrightarrow V,\quad k\in\mathbb{N},
\]
such that
\[
\sum_{i+j=k}(-)^{ij}\sum_{\sigma\in S_{i,j}}(-)^{\sigma}\alpha(\sigma
,\boldsymbol{v})\,\lambda_{j+1}(\lambda_{i}(v_{\sigma(1)},\ldots,v_{\sigma
(i)}),v_{\sigma(i+1)},\ldots,v_{\sigma(i+j)})=0,
\]
for all $\boldsymbol{v}=(v_{1},\ldots,v_{k})\in V^{\times k}$, $k\in
\mathbb{N}$. This is the classical notion of $L_{\infty}$ algebra \cite{ls93}.
However, I will refer instead to an equivalent notion where degrees are
shifted and the structure maps are graded symmetric, instead of graded
skew-symmetric. Following \cite{s09}, I call such an equivalent notion an
$L_{\infty}[1]$ algebra (see also \cite{v05}). Using $L_{\infty}[1]$ algebras
simplifies the signs in all formulas of this paper.

\begin{definition}
An $L_{\infty}[1]$ \emph{algebra} is a graded vector space $V$ equipped with a
family of $k$-ary, graded symmetric, multilinear, degree $1$ maps
\[
\lambda_{k}:{}V^{\times k}\longrightarrow V,\quad k\in\mathbb{N},
\]
such that
\[
\sum_{i+j=k}\sum_{\sigma\in S_{i,j}}\alpha(\sigma,\boldsymbol{v}%
)\,\lambda_{j+1}(\lambda_{i}(v_{\sigma(1)},\ldots,v_{\sigma(i)}),v_{\sigma
(i+1)},\ldots,v_{\sigma(i+j)})=0,
\]
for all $\boldsymbol{v}=(v_{1},\ldots,v_{k})\in V^{\times k}$, $k\in
\mathbb{N}$ (in particular, $(V,\lambda_{1})$ is a cochain complex).
\end{definition}

$L_{\infty}$ algebra structures on $V$ correspond bijectively to $L_{\infty
}[1]$ algebra structures on the suspension $V[1]:=\bigoplus_{i}V[1]^{i}$,
where $V[1]^{i}:=V^{i+1}$. The bijection is obtained by applying the
d\'{e}calage isomorphism (between exterior powers of $V$ and symmetric powers
of $V[1]$):
\[
\Lambda^{k}V\longrightarrow S^{k}V[1],\quad v_{1}\wedge\cdots\wedge
v_{k}\longmapsto(-)^{(k-1)\bar{v}_{1}+(k-2)\bar{v}_{2}+\cdots+\bar{v}_{k-1}%
}v_{1}\cdots v_{k},
\]
where $\bar{v}_{i}$ is the degree of $v_{i}$ in $V$.

Let $V$ be an $L_{\infty}[1]$ algebra. Definition below is the $L_{\infty}[1]$
version of a the definition of $L_{\infty}$-module \cite{lm95}.

\begin{definition}
An $L_{\infty}[1]$ \emph{module over }$V$ is a graded vector space $W$
equipped with a family of $k$-ary, multilinear, degree $1$ maps
\[
\mu_{k}:{}V^{\times(k-1)}\times W\longrightarrow W,\quad k\in\mathbb{N},
\]
which are graded symmetric in the first $k-1$ arguments, and such that
\begin{align}
&  \sum_{i+j=k}\sum_{\sigma\in S_{i,j}}\alpha(\sigma,\boldsymbol{v}%
)\,\mu_{j+1}(\lambda_{i}(v_{\sigma(1)},\ldots,v_{\sigma(i)}),v_{\sigma
(i+1)},\ldots,v_{\sigma(i+j-1)}|w)\nonumber\\
&  +\sum_{i+j=k}\sum_{\sigma\in S_{i,j}}(-)^{\chi}\,\alpha(\sigma
,\boldsymbol{v})\mu_{i+1}(v_{\sigma(1)},\ldots,v_{\sigma(i)}|\mu_{j}%
(v_{\sigma(i+1)},\ldots,v_{\sigma(i+j-1)}|w))=0 \label{Jac1}%
\end{align}
for all $\boldsymbol{v}=(v_{1},\ldots,v_{k-1})\in V^{\times(k-1)}$, $w\in W$,
$k\in\mathbb{N}$, where
$\chi=\bar{v}_{\sigma(1)}+\cdots+\bar{v}_{\sigma(i)}$ (in particular, $(W,\mu_{1})$ is a cochain complex).
\end{definition}

In a similar way one can write a definition of \emph{right} $L_{\infty}[1]$
\emph{module}, generalizing the standard notion of right Lie algebra module.
Since, apparently, such a definition does not appear in literature, I record
it here.

\begin{definition}
A \emph{right} $L_{\infty}[1]$ \emph{module over }$V$ is a graded vector space
$Z$ equipped with a family of $k$-ary, multilinear, degree $1$ maps
\[
\rho_{k}:{}V^{\times(k-1)}\times Z\longrightarrow Z,\quad k\in\mathbb{N},
\]
which are graded symmetric in the first $k-1$ arguments, and such that
\begin{align}
&  \sum_{i+j=k}\sum_{\sigma\in S_{i,j}}\alpha(\sigma,\boldsymbol{v}%
)\,\rho_{j+1}(\lambda_{i}(v_{\sigma(1)},\ldots,v_{\sigma(i)}),v_{\sigma
(i+1)},\ldots,v_{\sigma(i+j-1)}|z)\nonumber\\
&  -\sum_{i+j=k}\sum_{\sigma\in S_{i,j}}(-)^{\chi}\,\alpha(\sigma
,\boldsymbol{v})\rho_{i+1}(v_{\sigma(1)},\ldots,v_{\sigma(i)}|\rho
_{j}(v_{\sigma(i+1)},\ldots,v_{\sigma(i+j-1)}|z))=0 \label{Jac2}%
\end{align}
for all $\boldsymbol{v}=(v_{1},\ldots,v_{k-1})\in V^{\times(k-1)}$, $z\in Z$,
$k\in\mathbb{N}$, $\chi = \bar{v}_{\sigma (1)} + \cdots + \bar{v}_{\sigma (i)}$ (in particular, $(Z,\rho_{1})$ is a cochain complex).
\end{definition}

Notice the minus sign in front of the second summand of the left hand side of
(\ref{Jac2}), in contrast with Formula (\ref{Jac1}).

\section{Left and Right Representations of Lie-Rinehart Algebras\label{SecLR}}

\subsection{Lie-Rinehart Algebras\label{SecLRA}}

Lie-Rinehart algebras appear in various areas of Mathematics. In differential
geometry, they appear as spaces of sections of Lie algebroids. The prototype
of a Lie algebroid is the tangent bundle. Accordingly, vector fields on a
manifold form a Lie-Rinehart algebra. In its turn, the theory of Lie
algebroids proved to encode salient features of foliation theory, group action
theory, Poisson geometry, etc. In this section, I report those definitions
from the theory of Lie-Rinehart algebras that are relevant for the purposes of
the paper. For more details about Lie-Rinehart algebras, see \cite{h04} and
references therein.

A \emph{Lie-Rinehart (LR) algebra} is a pair $(A,L)$ where $A$ is an
associative, commutative, unital algebra over a field $K$ of zero
characteristic, and $L$ is a Lie algebra. Moreover, $L$ is an $A$-module and
$A$ is an $L$-module (with structure map $\alpha:L\longrightarrow
\mathrm{End}_{K}A$, called the \emph{anchor}). All these structures fulfills
the following compatibility conditions.\ For $a,b\in A$ and $\xi,\zeta\in L$
\begin{align*}
\alpha(\xi)(ab)  &  =\alpha(\xi)(a)b+a\alpha(\xi)(b)\\
(a\alpha(\xi))(b)  &  =a\alpha(\xi)(b)\\
\lbrack\xi,a\zeta]  &  =\alpha(\xi)(a)\zeta+a[\xi,\zeta].
\end{align*}
The first identity tells us that $L$ acts on $A$ by derivations. The second
identity tells us that the anchor $\alpha:L\longrightarrow\mathrm{Der}A$ is
$A$-linear. The third identity tells us that for all $\xi\in L$, the pair
$([\xi,{}\cdot{}],\alpha(\xi))$ is a \emph{derivation} of $L$. Recall that a
\emph{derivation} of an $A$-module $P$ is a pair $\mathbb{X}=(X,\sigma
_{\mathbb{X}})$ where $X:P\longrightarrow P$ is a $K$-linear operator and
$\sigma_{\mathbb{X}}$ is a derivation of $A$, called the \emph{symbol} of
$\mathbb{X}$, such that, for $a\in A$, and $p\in P$
\[
X(ap)=\sigma_{\mathbb{X}}(a)p+aX(p).
\]
Denote by $\mathrm{Der}P$ the set of derivations of $P$. Notice, for future
use, that there are two different $A$-module structures on $\mathrm{Der}P$.
The first one has structure map $(a,\mathbb{X})\longmapsto a^{L}%
\mathbb{X}:=(aX,a\sigma_{\mathbb{X}})$. The second one has structure map
$(a,\mathbb{X})\longmapsto a^{R}\mathbb{X}:=(X\circ a,a\sigma_{\mathbb{X}})$.
Here, $a$ is interpreted as the multiplication operator $A\longrightarrow A$,
$b\longmapsto ab$. Write $\mathrm{Der}^{L}P$ for $\mathrm{Der}P$ with the
first $A$-module structure, and $\mathrm{Der}^{R}P$ for $\mathrm{Der}P$ with
the second $A$-module structure.

The prototype of an LR algebra is the pair $(A,\mathrm{Der}A)$, with Lie
bracket the standard commutator of derivations, and anchor the identity. It is
easy to see that both $(A,\mathrm{Der}^{L}P)$ and $(A,\mathrm{Der}^{R}P)$ are
also LR algebras with Lie bracket the standard commutator, and anchor
$\mathbb{X}\longmapsto\sigma_{\mathbb{X}}$. As already outlined, in differential geometry LR
algebras appear as pairs $(A,L)$ where $A$ is the algebra of smooth real
functions on a smooth manifold $M$, and $L$ is the module of sections of a Lie
algebroid over $M$.

\subsection{Connections along Lie-Rinehart Algebras}

Connections along LR algebras are the algebraic counterparts of connections
along Lie algebroids \cite{elw99,x99}. Let $(A,L)$ be a LR algebra and $P,Q$
be $A$-modules. A \emph{left }$(A,L)$\emph{-connection} in $P$ (or, a
\emph{left connection along }$(A,L)$) is a map $\nabla:L\longrightarrow
\mathrm{End}_{K}P$, written $\xi\longmapsto\nabla_{\xi}$, such that, for $a\in
A$, $\xi\in L$, and $p\in P$,
\begin{align*}
\nabla_{\xi}(ap)  &  =\alpha(\xi)(a)p+a\nabla_{\xi}p,\\
\nabla_{a\xi}p  &  =a\nabla_{\xi}p.
\end{align*}
The first identity tells us that the pair $(\nabla_{\xi},\alpha(\xi))$ is a
derivation. The second identity tells us that the map $L\longrightarrow
\mathrm{Der}^{L}P$, $\xi\longmapsto(\nabla_{\xi},\alpha(\xi))$ is $A$-linear.
A left $(A,L)$-connection $\nabla$ is \emph{flat} if, for all $\zeta,\xi\in
L$,
\[
\lbrack\nabla_{\xi},\nabla_{\zeta}]-\nabla_{\lbrack\xi,\zeta]}=0,
\]
which tells us that that $\nabla$ is a homomorphism of Lie algebras. Left
connections along Lie-Rinehart algebras generalize the standard differential
geometry notion of linear connections in vector bundles.

Let $P$ be an $A$-module and $\nabla$ an $(A,L)$-connection in it. There is an
associated sequence
\begin{equation}
0\longrightarrow P\overset{D^{\nabla}}{\longrightarrow}\mathrm{Hom}%
_{A}(L,P)\overset{D^{\nabla}}{\longrightarrow}\cdots\longrightarrow
\mathrm{Alt}_{A}^{k}(L,P)\overset{D^{\nabla}}{\longrightarrow}\mathrm{Alt}%
_{A}^{k+1}(L,P)\overset{D^{\nabla}}{\longrightarrow}\cdots\label{27}%
\end{equation}
defined via the Chevalley-Eilenberg formula
\begin{align*}
(D^{\nabla}\Omega)(\xi_{1},\ldots,\xi_{k+1}):=  &  \sum_{i}(-)^{i+1}%
\nabla_{\xi_{i}}\Omega(\xi_{1},\ldots,\widehat{\xi_{i}},\ldots,\xi_{k+1})\\
&  +\sum_{i<j}(-)^{i+j}\Omega([\xi_{i},\xi_{j}],\xi_{1},\ldots,\widehat
{\xi_{i}},\ldots,\widehat{\xi_{j}},\ldots,\xi_{k+1}),
\end{align*}
where a hat $\ \widehat{\cdot}\ $ denotes omission, $\Omega\in\mathrm{Alt}_{A}%
^{k}(L,P)$ is an $A$-multilinear, alternating map $L^{\times k}\longrightarrow
P$, and $\xi_{1},\ldots,\xi_{k+1}\in L$. If $\nabla$ is flat, $D^{\nabla}$ is
a differential, i.e. $D^{\nabla}\circ D^{\nabla}=0$, and sequence (\ref{27})
is a (cochain) complex. Chevalley-Eilenberg cochain complexes of Lie algebras, de Rham
complexes of smooth manifolds, and, more generally, Chevalley-Eilenberg
complexes of Lie algebroids are of the kind (\ref{27}).

A right $(A,L)$\emph{-connection} in $Q$ (or, a \emph{right connection along
}$(A,L)$) is a map $\Delta:L\longrightarrow\mathrm{End}_{K}Q$, written
$\xi\longmapsto\Delta_{\xi}$, such that, for $a\in A$, $\xi\in L$, and $q\in
Q$,
\begin{align}
\Delta_{\xi}(aq)  &  =-\alpha(\xi)(a)q+a\Delta_{\xi}q.\label{16}\\
\Delta_{a\xi}q  &  =\Delta_{\xi}(aq). \label{17}%
\end{align}
Notice that the operators $\Delta_{\xi}$ are usually written as acting from
the right. I prefer to keep a different notation which is simpler to handle in
the graded case. Identity (\ref{16}) tells us that the pair $(\Delta_{\xi
},-\alpha(\xi))$ is a derivation. Identity (\ref{17}) tells us that the map
$L\longrightarrow\mathrm{Der}^{R}Q$, $\xi\longmapsto(\Delta_{\xi},-\alpha
(\xi))$ is $A$-linear. A right $(A,L)$-connection $\Delta$ is \emph{flat} if,
for all $\zeta,\xi\in L$,
\[
\lbrack\Delta_{\xi},\Delta_{\zeta}]+\Delta_{\lbrack\xi,\zeta]}=0,
\]
which tells us that $\Delta$ is an anti-homomorphism of Lie algebras.

Let $Q$ be an $A$-module and $\Delta$ an $(A,L)$-connection in it. There is an
associated sequence
\begin{equation}
0\longleftarrow Q\overset{D^{\Delta}}{\longleftarrow}L\otimes_{A}%
Q\longleftarrow\cdots\overset{D^{\Delta}}{\longleftarrow}\Lambda_{A}%
^{k}L\otimes_{A}Q\overset{D^{\Delta}}{\longleftarrow}\Lambda_{A}^{k+1}%
L\otimes_{A}Q\longleftarrow\cdots\label{28}%
\end{equation}
defined via the Rinehart formula \cite{r63}
\begin{align*}
D^{\Delta}(\xi_{1}\wedge\cdots\wedge\xi_{k+1}\otimes q):=  &  \sum
_{i<j}(-)^{i+j}[\xi_{i},\xi_{j}]\wedge\xi_{1}\wedge\cdots\wedge\widehat
{\xi_{i}}\wedge\cdots\wedge\widehat{\xi_{j}}\wedge\cdots\wedge\xi_{k+1}\otimes
q\\
&  +\sum_{i}(-)^{i+1}\xi_{1}\wedge\cdots\wedge\widehat{\xi_{i}}\wedge
\cdots\wedge\xi_{k+1}\otimes\Delta_{\xi_{i}}q,
\end{align*}
$\xi_{1},\ldots,\xi_{k+1}\in L$, $q\in Q$. If $\Delta$ is flat, $D^{\Delta}$
is a differential, i.e. $D^{\Delta}\circ D^{\Delta}=0$, and sequence
(\ref{28}) is a (chain) complex. The $\mathrm{Diff}$-Spencer complex (dual to
the first Spencer sequence) of a linear differential operator \cite{kv98} is
of the kind (\ref{28}). Another, closely related, motivation for considering
both left and right connections along LR algebras is the following. Let
$\mathcal{D}$ be the algebra of linear differential operators over $A$. In
general, $\mathcal{D}$ is a non-commutative algebra. There are obvious
inclusions $A$, $\mathrm{Der}A\subset\mathcal{D}$, and two different
$A$-module structures in $\mathcal{D}$ with structure maps $(a,\square
)\longmapsto a\circ\square$, and $(a,\square)\longmapsto\square\circ a$
respectively. Here, $a$ is interpreted as a differential operator of order
$0$: multiplication by $a$ operator. Write $\mathcal{D}^{L}$ for $\mathcal{D}$ with the first $A$-module
structure, and $\mathcal{D}^{R}$ for $\mathcal{D}$ with the second $A$-module
structure. Now let $\xi$ be a derivation of $A$. Define an operator
$\nabla_{\xi}:\mathcal{D}^{L}\longrightarrow\mathcal{D}^{L}$ (resp.,
$\Delta_{\xi}:\mathcal{D}^{R}\longrightarrow\mathcal{D}^{R}$), by putting
$\nabla_{\xi}\square:=\xi\circ\square$ (resp., $\Delta_{\xi}\square
:=\square\circ\xi$). Both $\nabla_{\xi}$ and $\Delta_{\xi}$ are derivations of
$A$-modules (beware, that neither $\nabla_{\xi}$, nor $\Delta_{\xi}$ is a
derivation of the algebra $\mathcal{D}$). Even more, $\nabla$ is a flat left
$(A,\mathrm{Der}A)$-connection in $\mathcal{D}^{L}$. Similarly, $\Delta$ is a
flat right $(A,\mathrm{Der}A)$-connection in $\mathcal{D}^{R}$. More
generally, there are canonical left and right connections in the universal
enveloping algebra of any LR algebra.

Notice that, under suitable regularity conditions on $L$, namely, $L$ being a
projective and finitely generated $A$-module of constant rank $n$, right
$(A,L)$-connections in an $A$-module $Q$ are actually equivalent to left
$(A,L)$-connections in $\Lambda_{A}^{n}L\otimes Q$. Moreover, the equivalence
identifies the \emph{Reinehart sequence }(\ref{28}) of $Q$ and the
\emph{Chevalley-Eilenberg sequence }(\ref{27}) of $\Lambda_{A}^{n}L\otimes
_{A}Q$ \cite{h99}. This is the case, for instance, when $L$ is the module of
sections of a (finite dimensional) Lie algebroid (over a connected manifold).
However, in the general case, right and left $(A,L)$-connections are distinct notions.

All the notions in this sections, in particular that of LR algebra, have a
graded analogue, which can be easily guessed exploiting the \emph{Koszul sign
rule}.

\section{Derivations and Multiderivations of Graded Modules\label{SecDer}}

\subsection{Derivations}

LR algebras have analogues up to homotopy, which are known as \emph{strong
homotopy} \emph{(SH) LR algebras} \cite{k01,h05,v12}. SH LR algebras were
introduced by Kjeseth in \cite{k01}, under the name \emph{homotopy
Lie-Rinehart pairs}, and appear naturally in different geometric contexts,
e.g., BRST-BV formalism \cite{k01b}, foliation theory \cite{h05,v12,v12b},
complex geometry \cite{y11}, action of $L_{\infty}$ algebras on graded
manifolds (see Section \ref{SecAct} of this paper). Kjeseth's definition of a
homotopy Lie-Rinehart pair makes use of the coalgebra concepts of
\emph{subordinate derivation sources}, and \emph{resting coderivations}. In a
similar spirit, Huebschmann proposed an equivalent definition making use of
\emph{coderivations}, and \emph{twisting cochains} \cite{h13}. In this paper,
I propose a third, equivalent (but, perhaps, somewhat more transparent)
definition in terms of multiderivations of graded modules. I summarize the
relevant facts about multiderivations in this Section. Propositions
\ref{Prop6}, \ref{Prop7}, \ref{Prop8}, \ref{Prop9} will play a key role in the sequel.

Let $A$ be an associative, graded commutative, unital algebra, and let $P,Q$
be $A$-modules. If there is risk of confusion, I will use the name
$A$\emph{-module derivation} for a \emph{derivation of an }$A$\emph{-module}
to make clear the distinction with derivations of algebras. {In Section
\ref{SecLRA}, I recalled the non-graded definition. The graded definition is
obtained, as usual, by applying the Koszul sign rule.} Namely, a graded\emph{
}$A$\emph{-module derivation} of $P$ (or, simply, a \emph{derivation}, if
there is no risk of confusion) is a pair $\mathbb{X}=(X,\sigma_{\mathbb{X}})$,
where $\sigma_{\mathbb{X}}$ is a graded derivation of $A$, called the
\emph{symbol of }$\mathbb{X}$, and $X$ is a $K$-linear, graded operator
$X:P\longrightarrow P$ such that
\[
X(ap)=(-)^{\mathbb{X}a}aX(p)+\sigma_{\mathbb{X}}(a)p,\quad a\in A,\ p\in P.
\]
Notice that, in general, $X$ does not determine $\sigma_{\mathbb{X}}$
uniquely. That is the reason why I added the datum of $\sigma_{\mathbb{X}}$ to
the definition of a derivation. However, if $P$ is a faithful module,
$\sigma_{\mathbb{X}}$ is determined by $X$ and one can identify $\mathbb{X}$
with its first component $X$. Accordingly, I will sometimes write
$\mathbb{X}(p)$ for $X(p)$ and use other similar slight abuses of notation
without further comment. Beware that \emph{an }$A$\emph{-module derivation
of }$A$\emph{ is not an ordinary derivation, in general}. Rather, it is a
first order differential operator.

\begin{example}
Let $\varphi:P\longrightarrow P$ be an $A$-linear map. Then $(\varphi,0)$ is a derivation.
\end{example}

I will denote by $\mathrm{Der}_{A}P$ (or simply $\mathrm{Der}P$ if there is no
risk of confusion) the set of left derivations of $P$. There are two different
$A$-module structures on $\mathrm{Der}P$. The first one has structure map
$(a,\mathbb{X})\longmapsto(aX,a\sigma_{\mathbb{X}})$. The second one has
structure map $(a,\mathbb{X})\longmapsto((-)^{a\mathbb{X}}X\circ
a,a\sigma_{\mathbb{X}})$. Here, $a$ is interpreted as the multiplication
operator $P\longrightarrow P$, $p\longmapsto ap$. Write $\mathrm{Der}_{A}%
^{L}P$ (or simply $\mathrm{Der}^{L}P$) for $\mathrm{Der}P$ with the first
$A$-module structure, and $\mathrm{Der}_{A}^{R}P$ (or simply $\mathrm{Der}%
^{R}P$) for $\mathrm{Der}P$ with the second $A$-module structure. Both
$(A,\mathrm{Der}^{L}P)$ and $(A,\mathrm{Der}^{R}P)$ are graded LR algebras
with Lie bracket given by%
\[
\lbrack\mathbb{X},\mathbb{X}^{\prime}]:=([X,X^{\prime}],[\sigma_{\mathbb{X}%
},\sigma_{\mathbb{X}^{\prime}}]),
\]
and anchor $\mathbb{X}\longmapsto\sigma_{\mathbb{X}}$.

Derivations of $A$-modules can be extended to tensor products and
homomorphisms as follows. Let $P$, $P^{\prime}$ be $A$-modules and let
$\mathbb{X}$, and $\mathbb{X}^{\prime}$ be derivations of $P$, and $P^{\prime
}$ respectively. Suppose that $\mathbb{X}$, and $\mathbb{X}^{\prime}$ share
the same symbol $\sigma=\sigma_{\mathbb{X}}=\sigma_{\mathbb{X}^{\prime}}$. It
is easy to see that the operator $X^{\otimes}:P\otimes_{K}P^{\prime
}\longrightarrow P\otimes_{K}P^{\prime}$ defined by
\[
X^{\otimes}(p\otimes p^{\prime}):=(Xp)\otimes p^{\prime}+(-)^{\sigma
p}p\otimes(X^{\prime}p^{\prime})
\]
descends to a well defined operator on $P\otimes_{A}P^{\prime}$, which,
abusing the notation, I denote again by $X^{\otimes}$. Moreover $\mathbb{X}%
^{\otimes}:=(X^{\otimes},\sigma)$ is a derivation. Similarly, the operator
$X^{\mathrm{Hom}}:\mathrm{Hom}_{K}(P,P^{\prime})\longrightarrow\mathrm{Hom}%
_{K}(P,P^{\prime})$ defined as
\[
(X^{\mathrm{Hom}}\varphi)(p):=X^{\prime}\varphi(p)-(-)^{\sigma\varphi}%
\varphi(Xp)
\]
descends to an operator $X^{\mathrm{Hom}}$ on $\mathrm{Hom}_{A}(P,P^{\prime})$
and $\mathbb{X}^{\mathrm{Hom}}:=(X^{\mathrm{Hom}},\sigma)$ is a derivation.

\subsection{Multiderivations}

Now, I generalize the notion of derivation to that of multiderivation. First I
discuss multiderivations of algebras.

\begin{definition}
A \emph{multiderivation of }$A$\emph{ with }$k$\emph{ entries} is a graded
symmetric, $K$-linear operator $H:A^{\times k}\longrightarrow A$, such that
\[
H(a_{1},\ldots,a_{k-1},ab)=(-)^{\chi}a\cdot H(a_{1},\ldots,a_{k-1}%
,b)+H(a_{1},\ldots,a_{k-1},a)\cdot b,
\]
{where $\chi=(\bar{H}+\bar{a}_{1}+\cdots+\bar{a}_{k-1})\bar{a}$,} for all
$a_{1},\ldots,a_{k-1},a,b\in A$.
\end{definition}

Denote by $\mathrm{Der}^{k}A$ the set of multiderivations of $A$ with $k$
entries. In particular, $\mathrm{Der}^{0}A=A$ and $\mathrm{Der}^{1}A$ consists
of standard graded derivations of $A$. Clearly, $\mathrm{Der}^{k}A$ is a
graded $A$-module. Put $\mathrm{Der}^{\bullet}A:=\bigoplus_{k}\mathrm{Der}%
^{k}A$, which is naturally bi-graded. However, the total degree will be of
primary importance for the purposes of this paper. The $A$-module
$\mathrm{Der}^{\bullet}A$ can be given the structure of a graded Lie algebra
(beware, not bi-graded) as follows. For $H\in\mathrm{Der}^{k}A$, and
$H^{\prime}\in\mathrm{Der}^{\ell}A$, let $[H,H^{\prime}]\in\mathrm{Der}%
^{k+\ell-1}A$ be defined by%
\[
\lbrack H,H^{\prime}]:=H\circ H^{\prime}-(-)^{HH^{\prime}}H^{\prime}\circ H,
\]
where $H\circ H^{\prime}$ is given by
\begin{equation}
(H\circ H^{\prime})(a_{1},\ldots,a_{k+\ell-1}):=\sum_{\sigma\in S_{\ell,k-1}%
}\alpha(\sigma,\boldsymbol{a})H(H^{\prime}(a_{\sigma(1)},\ldots,a_{\sigma
(\ell)}),a_{\sigma(\ell+2)},\ldots,a_{\sigma(k+\ell-1)}), \label{12}%
\end{equation}
$\boldsymbol{a}=(a_{1},\ldots,a_{k+\ell-1})\in A^{\times(k+\ell-1)}$. Formulas
of the kind (\ref{12}) will often appear below. Apparently, this kind of
formulas first appeared in \cite{g63} (for the case of a, generically non
commutative, ring). Accordingly, I will refer to them as
\emph{Gerstenhaber-type fomulas}, without further comment.

Now, let $L$ be an $A$-module.

\begin{definition}
\label{Def1}An $A$\emph{-module multiderivation of }$L$ (or, simply, a
\emph{multiderivation}, if there is no risk of confusion)\emph{ with }%
$k$\emph{ entries} is a pair $\mathbb{X}=(X,\sigma_{\mathbb{X}})$ where
$\sigma_{\mathbb{X}}$ is a graded symmetric, $A$-multilinear map
$\sigma_{\mathbb{X}}:L^{\times(k-1)}\longrightarrow\mathrm{Der}A$, called the
\emph{symbol} of $\mathbb{X}$, and $X$ is a graded symmetric, $K$-multilinear
map $X:L^{\times k}\longrightarrow L$, such that
\[
X(\xi_{1},\ldots,\xi_{k-1},a\xi_{k})=\sigma_{\mathbb{X}}(\xi_{1},\ldots
,\xi_{k-1}|a)\cdot\xi_{k}+(-)^{\chi^{\prime}}a\cdot X(\xi_{1},\ldots,\xi
_{k}),
\]
where $\chi^{\prime}=(\bar{X}+\bar{\xi}_{1}+\cdots+\bar{\xi}_{k-1})\bar{a}$,
and I put $\sigma_{\mathbb X}(\xi_{1},\ldots,\xi_{k-1}|a):= \sigma_{\mathbb X}(\xi_{1},\ldots,\xi_{k-1})(a)$, for
all $\xi_{1},\ldots,\xi_{k}\in L$, $a\in A$.
\end{definition}

\begin{example}
Let $\mathbb{X}=(X,\sigma_{\mathbb{X}})$ be an $A$-module multiderivation. If
$A=K$, then $\sigma_{\mathbb{X}}=0$ and $X$ is simply a $K$-multilinear map.
Conversely, every $K$-multilinear map is a $K$-module multiderivation.
\end{example}

A version of Definition \ref{Def1} appeared in \cite{cm08} (Section 2.1).
However, in that paper, the authors consider skew-symmetric multi-derivations
of ungraded modules of smooth sections of vector bundles. I consider symmetric
multiderivations for convenience. One can pass from the latter to the former
via suitable d\'{e}calage isomorphisms.

\subsection{Lie algebras of multiderivations}

Let $\mathrm{Der}_{A}^{k}L$ (or simply $\mathrm{Der}^{k}L$) denote the set of
multiderivations of $L$ with $k$ entries (beware that, in \cite{cm08}, the
authors denote by $\mathrm{Der}^{k}$ skew-symmetric multiderivations with
$k+1$ entries). In particular, $\mathrm{Der}^{0}L=L$ and $\mathrm{Der}%
^{1}L=\mathrm{Der}L$. Clearly, $\mathrm{Der}^{k}L$ is a graded $A$-module. Put
$\mathrm{Der}_{A}^{\bullet}L\equiv\mathrm{Der}^{\bullet}L:=\bigoplus
_{k}\mathrm{Der}_{A}^{k}L$, which is naturally bi-graded. The $A$-module
$\mathrm{Der}^{\bullet}L$ can be given the structure of a graded (not
bi-graded) Lie algebra as follows. For $\mathbb{X}$ a multiderivation with $k$
entries, and $\mathbb{Y}$ a multiderivations with $\ell$ entries, let
$[\mathbb{X},\mathbb{Y}]$ be the multiderivation with $k+\ell-1$ entries
defined by%
\[
\lbrack\mathbb{X},\mathbb{Y}]:=(\mathbb{[}X,Y],\sigma_{\lbrack\mathbb{X}%
,\mathbb{Y}]}),
\]
where%
\[
\lbrack X,Y]:=X\circ Y-(-)^{XY}Y\circ X,
\]
$X\circ Y$ being given by a Gerstenhaber-type formula (\ref{12}), and
\[
\sigma_{\lbrack\mathbb{X},\mathbb{Y}]}:=\sigma_{\mathbb{X}}\circ
Y-(-)^{\mathbb{XY}}\sigma_{\mathbb{Y}}\circ X+[\sigma_{\mathbb{X}}%
,\sigma_{\mathbb{Y}}]
\]
where $\sigma_{\mathbb{X}}\circ Y$ is given again by a Gerstenhaber-type
formula, and $[\sigma_{\mathbb{X}},\sigma_{\mathbb{Y}}]$ is given by%
\begin{equation}
\lbrack\sigma_{\mathbb{X}},\sigma_{\mathbb{Y}}](\xi_{1},\ldots,\xi_{k+\ell
-2}):=\sum_{\sigma\in S_{k-1,\ell-1}}(-)^{\chi}\alpha(\sigma,\boldsymbol{\xi
})[\sigma_{\mathbb{X}}(\xi_{\sigma(1)},\ldots,\xi_{\sigma(k-1)}),\sigma
_{\mathbb{Y}}(\xi_{\sigma(k)},\ldots,\xi_{\sigma(k+\ell-2)})], \label{13}%
\end{equation}
with $\chi=\mathbb{\bar{Y}}(\bar{\xi}_{\sigma(1)}+\cdots+\bar{\xi}%
_{\sigma(k-1)})$, and $\boldsymbol{\xi}=(\xi_{1},\ldots,\xi_{k+\ell-2})\in
L^{\times(k+\ell-2)}$.

Now, consider the graded commutative algebra $\mathrm{Sym}_{A}(L,A)$ of graded
\emph{symmetric forms} on $L$, i.e., $A$-multilinear, graded symmetric maps
$L\times\cdots\times L\longrightarrow A$. Consider also the symmetric algebra
$S_{A}^{\bullet}L$ of $L$. I will refer to elements in $S_{A}^{\bullet}L$ as
\emph{symmetric tensors} (or just \emph{tensors}). See Appendix
\ref{SecMultAlg} for notations about forms and tensors and structures on them
relevant for the purposes of this paper.

\begin{proposition}
\label{Prop6}There is a canonical morphism of graded Lie algebras
\[
\eta:\mathrm{Der}^{\bullet}L\longrightarrow\{\text{derivations of
}\mathrm{Sym}_{A}(L,A)\},
\]
such that, $\eta$ maps multiderivations with $k$ entries to derivations taking
$\ell$-forms to $(\ell+k-1)$-forms. Moreover, if $L$ is projective and
finitely generated, then $\eta$ is an isomorphism.
\end{proposition}

\begin{proof}
Let $\mathbb{X}$ be a multiderivation with $k$ entries and $\omega$ an $\ell
$-form. Put
\[
\eta(\mathbb{X})(\omega):=\sigma_{\mathbb{X}}\circ\omega-(-)^{\mathbb{X}%
\omega}\omega\circ X,
\]
where $\omega\circ X$ is given by a Gerstenhaber-type formula and
\begin{equation}
(\sigma_{\mathbb{X}}\circ\omega)(\xi_{1},\ldots,\xi_{\ell+k-1}):=\sum
_{\sigma\in S_{k-1,\ell}}(-)^{\chi}\alpha(\sigma,\boldsymbol{\xi}%
)\sigma_{\mathbb{X}}(\xi_{\sigma(1)},\ldots,\xi_{\sigma(k-1)}|\omega
(\xi_{\sigma(k)},\ldots,\xi_{\sigma(\ell+k-1)})), \label{15}%
\end{equation}
with $\chi=\bar{\omega}(\bar{\xi}_{\sigma(1)}+\cdots+\bar{\xi}_{\sigma(k-1)}%
)$, and $\boldsymbol{\xi}=(\xi_{1},\ldots,\xi_{\ell+k-1})\in L^{\times
(\ell+k-1)}$. A careful but straightforward computation shows that $\eta$ is a
well defined morphism of graded Lie algebras.

Now, let $L$ be projective and finitely generated. Then $L\simeq L^{\ast\ast}$
and an inverse homomorphism $\eta^{-1}$ is implicitly defined as follows. Let
$D$ be a derivation of $\mathrm{Sym}_{A}(L,A)$ taking $\ell$-forms to
$(\ell+k-1)$-forms, and let $\omega\in L^{\ast}=\mathrm{Hom}_{A}(L,A)$ be a
$1$-form. Put $\eta^{-1}(D):=(X_{D},\sigma_{D})$, where
\[
\sigma_{D}(\xi_{1},\ldots,\xi_{k-1}|a):=(-)^{\chi}(Da)(\xi_{1},\ldots
,\xi_{k-1}),
\]
with $\chi=(\bar{\xi}_{1}+\cdots+\bar{\xi}_{k-1})\bar{a}$, and
\[
\omega(X_{D}(\xi_{1},\ldots,\xi_{k})):=\sum_{i=1}^{k}(-)^{\omega
D+\chi^{\prime}}\sigma_{D}(\xi_{1},\ldots,\widehat{\xi_{i}},\ldots,\xi
_{k}|\omega(\xi_{i}))+(-)^{\omega D}D(\omega)(\xi_{1},\ldots,\xi_{k}),
\]
with $\chi^{\prime}=\bar{\omega}(\bar{\xi}_{1}+\cdots+\widehat{\bar{\xi}_{i}%
}+\cdots+\bar{\xi}_{k})+\bar{\xi}_{i}(\bar{\xi}_{i+1}+\cdots+\bar{\xi}_{k})$,
and $\xi_{1},\ldots,\xi_{k}\in L$.
\end{proof}

\begin{proposition}
\label{Prop7}There is a canonical inclusion of graded Lie algebras
\[
\nu:\mathrm{Der}^{\bullet}L\longrightarrow\{\text{multiderivations of }%
S_{A}^{\bullet}L\},
\]
such that $\nu$ maps surjectively $k$-entry multiderivations to $k$-entry
multiderivations taking $S_{A}^{\ell_{1}}L\times\cdots\times S_{A}^{\ell_{k}%
}L$ to $S_{A}^{\ell_{1}+\cdots+\ell_{k}+k-1}L$.
\end{proposition}

\begin{proof}
Let $\mathbb{X}$ be a multiderivation of $L$. It is easy to see that
$\mathbb{X}$ can be uniquely extended to $S_{A}^{\bullet}L$ as a multiderivation.
\end{proof}

\subsection{Derivation valued symmetric forms}

Now, let $P,Q$ be $A$-modules, and $\mathcal{L}^{k}(P)$ be the set of pairs
$(\mathbb{X},\nabla)$, where $\mathbb{X}$ is a multiderivation of $L$ with $k$
entries, and $\nabla$ is a $\mathrm{Der}^{L}P$-valued $(k-1)$-form, i.e., a
graded symmetric, $A$-multilinear map $\nabla:L^{\times(k-1)}\longrightarrow
\mathrm{Der}^{L}P$, such that, for all $\xi_{1},\ldots,\xi_{k-1}\in L$
\begin{equation}
\sigma_{\nabla(\xi_{1},\ldots,\xi_{k-1})}=\sigma_{\mathbb{X}}(\xi_{1}%
,\ldots,\xi_{k-1}). \label{29}%
\end{equation}
In other words
\[
\nabla(\xi_{1},\ldots,\xi_{k-1}|ap)=(-)^{\chi}a\nabla(\xi_{1},\ldots,\xi
_{k-1}|p)+\sigma_{\mathbb{X}}(\xi_{1},\ldots,\xi_{k-1}|a)p,
\]
where $\chi=(\bar{\nabla}+\bar{\xi}_{1}+\cdots+\bar{\xi}_{k-1})\bar{a}$, $a\in
A$, $p\in P$. Put $\mathcal{L}(P):=\bigoplus_{k}\mathcal{L}^{k}(P)$.

Similarly, let $\mathcal{R}^{k}(Q)$ be the set of pairs $(\mathbb{X},\Delta)$,
where $\mathbb{X}$ is as above and $\Delta$ is a $\mathrm{Der}^{R}Q$-valued
$(k-1)$-form, i.e., a graded symmetric, $A$-multilinear map $\Delta
:L^{\times(k-1)}\longrightarrow\mathrm{Der}^{R}Q$, such that, for all $\xi
_{1},\ldots,\xi_{k-1}\in L$
\begin{equation}
\sigma_{\Delta(\xi_{1},\ldots,\xi_{k-1})}=-\sigma_{\mathbb{X}}(\xi_{1}%
,\ldots,\xi_{k-1}). \label{18}%
\end{equation}
In other words
\[
\Delta(\xi_{1},\ldots,\xi_{k-1}|aq)=(-)^{\chi^{\prime}}a\Delta(\xi_{1}%
,\ldots,\xi_{k-1}|q)-\sigma_{\mathbb{X}}(\xi_{1},\ldots,\xi_{k-1}|a)q,
\]
where $\chi^{\prime}=(\bar{\Delta}+\bar{\xi}_{1}+\cdots+\bar{\xi}_{k-1}%
)\bar{a}$, $a\in A$, $q\in Q$. Notice the minus sign in the right hand side of
(\ref{18}), in contrast with Formula (\ref{29}). Put $\mathcal{R}%
(Q):=\bigoplus_{k}\mathcal{R}^{k}(Q)$.

Both $\mathcal{L}(P)$ and $\mathcal{R}(Q)$ can be given a structure of graded
(not bi-graded) Lie algebra as follows. For $(\mathbb{X},\nabla)\in
\mathcal{L}^{k}(P)$, and $(\mathbb{X}^{\prime},\nabla^{\prime})\in
\mathcal{L}^{\ell}(P)$, let $[(\mathbb{X},\nabla),(\mathbb{X}^{\prime}%
,\nabla^{\prime})]\in\mathcal{L}^{k+\ell-1}(P)$ be defined by%
\[
\lbrack(\mathbb{X},\nabla),(\mathbb{X}^{\prime},\nabla^{\prime}%
)]:=(\mathbb{[X},\mathbb{Y}],\nabla^{\prime\prime}),
\]
with
\begin{equation}
\nabla^{\prime\prime}:=\nabla\circ X^{\prime}-(-)^{\mathbb{XX}^{\prime}}%
\nabla^{\prime}\circ X+[\nabla,\nabla^{\prime}] \label{30}%
\end{equation}
where $\nabla\circ Y$ is given by a Gerstenhaber-type formula, and
$[\nabla,\nabla^{\prime}]$ is given by a similar formula as (\ref{13}).
Similarly, for $(\mathbb{X},\Delta)\in\mathcal{R}^{k}(P)$, and $(\mathbb{X}%
^{\prime},\Delta^{\prime})\in\mathcal{R}^{\ell}(P)$, let $[(\mathbb{X}%
,\Delta),(\mathbb{X}^{\prime},\Delta^{\prime})]\in\mathcal{R}^{k+\ell-1}(P)$
be defined by%
\[
\lbrack(\mathbb{X},\Delta),(\mathbb{X}^{\prime},\Delta^{\prime}%
)]:=(\mathbb{[X},\mathbb{Y}],\Delta^{\prime\prime}),
\]
with
\begin{equation}
\Delta^{\prime\prime}:=\Delta\circ X^{\prime}-(-)^{\mathbb{XX}^{\prime}}%
\Delta^{\prime}\circ X-[\Delta,\Delta^{\prime}]. \label{19}%
\end{equation}
Notice the minus sign in front of the third summand of the right hand side of
(\ref{19}), in contrast with Formula (\ref{30}).

I will say that an element $(\mathbb{X},\nabla)$ of $\mathcal{L}(P)$ (resp.,
$\mathcal{R}(P)$) is \emph{subordinate} to the multiderivation $\mathbb{X}$.

\begin{proposition}
\label{Prop8}There is a canonical morphism of graded Lie algebras
\[
\eta^{L}:\mathcal{L}(P)\longrightarrow\{\mathrm{Sym}_{A}(L,A)\text{-module
derivations of }\mathrm{Sym}_{A}(L,P)\}
\]
such that, for $(\mathbb{X},\nabla)\in\mathcal{L}^{k}(P)$, $\eta
^{L}(\mathbb{X},\nabla)$ takes $\ell$-forms to $(\ell+k-1)$-forms, and the
symbol of $\eta^{L}(\mathbb{X},\nabla)$ is $\eta(\mathbb{X})$. Moreover, if
$L$ is projective and finitely generated, $\eta^{L}$ is an isomorphism.
\end{proposition}

\begin{proof}
Let $(\mathbb{X},\nabla)\in\mathcal{L}(P)$. Put $\eta^{L}(\mathbb{X}%
,\nabla):=(D,\eta(\mathbb{X}))$, where, for any $P$-valued form $\Omega$,
\[
D(\Omega):=\nabla\circ\Omega-(-)^{\mathbb{X}\Omega}\Omega\circ X,
\]
with $\Omega\circ X$ being given by a Gerstenhaber-type formula, and
$\nabla\circ\Omega$ being given by a similar formula as (\ref{15}). A careful
but straightforward computation shows that $\eta^{L}$ is a well defined
morphism of graded Lie algebras.

If $L$ is projective and finitely generated, $\eta$ is invertible and one can
define $(\eta^{L})^{-1}$ as follows. Let $\mathbb{D}=(D,\sigma_{\mathbb{D}})$
be a $\mathrm{Sym}_{A}(L,A)$-module derivation of $\mathrm{Sym}_{A}(L,P)$ such
that $D$ takes $\ell$-forms to $(\ell+k-1)$-forms. Put
\[
(\eta^{L})^{-1}(\mathbb{D}):=(\eta^{-1}(\sigma_{\mathbb{D}}),\nabla_{D}),
\]
where
\[
\nabla_{D}(\xi_{1},\ldots,\xi_{k-1}|p):=(-)^{\chi}(Dp)(\xi_{1},\ldots
,\xi_{k-1}),
\]
where $\chi=(\bar{\xi}_{1}+\cdots+\bar{\xi}_{k-1})\bar{p}$, $\xi_{1}%
,\ldots,\xi_{k-1}\in L$, and $p\in P$.
\end{proof}

\begin{proposition}
\label{Prop9}There is a canonical morphism of graded Lie algebras
\[
\eta^{R}:\mathcal{R}(Q)\longrightarrow\{\mathrm{Sym}_{A}(L,A)\text{-module
derivations of }S_{A}^{\bullet}L\otimes_{A}Q\}
\]
such that, for $(\mathbb{X},\Delta)\in\mathcal{R}^{k}(Q)$, $\eta
^{R}(\mathbb{X},\Delta)$ takes $S_{A}^{\ell}L\otimes_{A}Q$ to $S_{A}%
^{\ell-k+1}L\otimes_{A}Q$, and the symbol of $\eta^{R}(\mathbb{X},\Delta)$ is
$\eta(\mathbb{X})$. Moreover, if $L$ is projective and finitely generated,
$\eta^{R}$ is an isomorphism.
\end{proposition}

\begin{proof}
Let $(\mathbb{X},\Delta)\in\mathcal{R}^{k}(Q)$. Put $\eta^{R}(\mathbb{X}%
,\Delta):=(D,\eta(\mathbb{X}))$, with
\begin{align*}
D(\xi_{1}\cdots\xi_{\ell}\otimes q):=  &  \sum_{\tau\in S_{k, \ell-k}}%
\alpha(\tau,\boldsymbol{\xi})X(\xi_{\tau(1)},\ldots,\xi_{\tau(k)})\xi
_{\tau(k+1)}\cdots\xi_{\tau(\ell)}\otimes q\\
&  -\sum_{\sigma\in S_{\ell - k+1,k-1}}(-)^{\chi}\alpha(\sigma,\boldsymbol{\xi
})\xi_{\sigma(1)}\cdots\xi_{\sigma(\ell-k+1)}\otimes\Delta(\xi_{\sigma
(\ell- k +2)},\ldots,\xi_{\sigma(\ell)}|q)
\end{align*}
where I put $\Delta(\xi_{1},\ldots,\xi_{\ell}|q):=\Delta(\xi_{1},\ldots
,\xi_{\ell})(q)$, and $\chi=\mathbb{\bar{X}(}\bar{\xi}_{\sigma(1)}+\cdots
+\bar{\xi}_{\sigma(\ell - k+1)})$, $q\in Q$, $\xi_{1},\ldots,\xi_{\ell}\in L$. A
careful but straightforward computation shows that $\eta^{R}$ is a well
defined morphism of graded Lie algebras.

If $L$ is projective and finitely generated, $\eta$ is invertible and one can
define $(\eta^{R})^{-1}$ as follows. Let $\mathbb{D}=(D,\sigma_{\mathbb{D}})$
be a $\mathrm{Sym}_{A}(L,A)$-module derivation of $S_{A}^{\bullet}L\otimes
_{A}Q$ such that $D$ takes $Q\otimes_{A}S_{A}^{\ell}L$ to $Q\otimes_{A}%
S_{A}^{\ell-k}L$. Put
\[
(\eta^{R})^{-1}(\mathbb{D}):=(\eta^{-1}(\sigma_{\mathbb{D}}),\Delta_{D}),
\]
where
\[
\Delta_{D}(\xi_{1},\ldots,\xi_{k}|q):=D(\xi_{1}\cdots\xi_{k}\otimes q),
\]
$\xi_{1},\ldots,\xi_{k}\in L$, $q\in Q$.
\end{proof}

\begin{remark}
\label{Rem1}Both $\eta^{L}$ and $\eta^{R}$ are bijective when restricted to
elements with fixed first component $\mathbb{X}\in\mathrm{Der}_{A}^{\bullet}L$
in the domain, and derivations with symbol equal to $\eta(\mathbb{X})$ in the codomain.
\end{remark}

\begin{remark}
In the following, it will be useful to consider suitable \textquotedblleft
completions\textquotedblright\ of some of the graded spaces that appeared in
this section. Namely, put
\begin{align*}
\widehat{\mathrm{Der}}{}^{\bullet}A  &  :=%
{\textstyle\prod\nolimits_{k}}
\mathrm{Der}^{k}A,\\
\widehat{\mathrm{Der}}{}_{A}^{\bullet}L\equiv\widehat{\mathrm{Der}}{}%
^{\bullet}L  &  :=%
{\textstyle\prod\nolimits_{k}}
\mathrm{Der}^{k}L,
\end{align*}
and, similarly, put
\begin{align*}
\widehat{\mathcal{L}}(P)  &  :=%
{\textstyle\prod\nolimits_{k}}
\mathcal{L}^{k}(P),\\
\widehat{\mathcal{R}}(Q)  &  :=%
{\textstyle\prod\nolimits_{k}}
\mathcal{R}^{k}(Q).
\end{align*}
For instance, an element $\mathbb{X}$ in $\widehat{\mathrm{Der}}{}%
_{A}^{\bullet}L$ is a formal infinite sum%
\[
\mathbb{X}=\mathbb{X}_{0}+\mathbb{X}_{1}+\mathbb{X}_{2}+\cdots=\sum
_{k=0}^{\infty}\mathbb{X}_{k}%
\]
of $A$-module multiderivations, such that $\mathbb{X}_{k}$ has exactly
$k$-entries. Similarly, an element $(\mathbb{X},\nabla)$ in $\widehat
{\mathcal{L}}(P)$ (resp., $\widehat{\mathcal{R}}(Q)$) is a formal infinite
sum
\[
(\mathbb{X},\nabla)=(\mathbb{X}_{0},\nabla_{0})+(\mathbb{X}_{1},\nabla
_{1})+(\mathbb{X}_{2},\nabla_{2})+\cdots=\sum_{k=0}^{\infty}(\mathbb{X}%
_{k},\nabla_{k})
\]
of elements in $\mathcal{L}(P)$ (resp., $\mathcal{R}(Q)$), such that
$(\mathbb{X}_{k},\nabla_{k})$ has $k$ entries. All results in this section
extend trivially to the above completions. For instance, The Lie brackets of
$\mathrm{Der}_{A}^{\bullet}L$, $\mathcal{L}(P)$ and $\mathcal{R}(P)$, extend
to $\widehat{\mathrm{Der}}{}_{A}^{\bullet}L$, $\widehat{\mathcal{L}}(P)$,
$\widehat{\mathcal{R}}(P)$ respectively, and the morphism $\eta$ extends to a
bracket preserving map
\[
\eta:\widehat{\mathrm{Der}}{}_{A}^{\bullet}L\longrightarrow
\{\text{\textquotedblleft formal\textquotedblright\ derivations of
}\mathrm{Sym}_{A}(L,A)\},
\]
where, by a \emph{\textquotedblleft formal\textquotedblright\ derivation} $D$
of $\mathrm{Sym}_{A}(L,A)$, I mean a formal infinite sum
\[
D=D_{0}+D_{1}+D_{2}+\cdots=\sum_{k=0}^{\infty}D_{k}%
\]
of standard derivations such that $D_{k}$ maps $\ell$-forms to $(\ell
+k-1)$-forms. Morphisms $\eta^{L}$ and $\eta^{R}$ extend in a similar way. I
leave the obvious details to the reader. In the following, if there is no risk
of confusion, I will refer to elements of $\widehat{\mathrm{Der}}{}^{\bullet
}A$ and $\widehat{\mathrm{Der}}{}^{\bullet}L$ simply as multiderivations.
Similarly, I will refer to \textquotedblleft formal\textquotedblright%
\ derivations of $\mathrm{Sym}_{A}(L,A)$ simply as derivations. The careful
reader will find even more uninfluential abuses of notations analogous to
these ones scattered in the text. For the sake of readability, I will not
comment further on them.
\end{remark}

\section{SH Lie-Rinehart Algebras\label{SecSHLR}}

Let $A$ be an associative, graded commutative, unital $K$-algebra, $L$ an
$A$-module, and let $\mathbb{X}=\mathbb{X}_{0}+\mathbb{X}_{1}+\mathbb{X}%
_{2}+\cdots\in\widehat{\mathrm{Der}}{}^{\bullet}L$ be a multiderivation of $L$.

\begin{definition}
\label{DefSHLR}A \emph{SH LR algebra}, or an $LR_{\infty}[1]$\emph{
algebra},\emph{ }is a pair $(A,L)$, equipped with a degree $1$,
multiderivation $\mathbb{X}$ such that, $\mathbb{X}_{0}=0$ and the
\emph{higher Jacobiator} $\mathbb{J}(\mathbb{X}):=\tfrac{1}{2}[\mathbb{X}%
,\mathbb{X}]$ vanishes.
\end{definition}

\begin{example}
Let $V$ be a graded $K$-vector space. $LR_{\infty}[1]$ algebra structures on
$(K,V)$ are equivalent to $L_{\infty}[1]$ algebra structures on $V$.
\end{example}

\begin{example}
$L_{\infty}$ algebroids \cite{sss09,b11} {(see also \cite{sz11,bp12})} provide
examples of SH LR algebras. Indeed, an $L_{\infty}$ algebroid is a graded
vector bundle $\mathcal{E}$ over a \emph{non-graded} smooth manifold $M$,
equipped with a SH LR algebra structure on $(A,L):=(C^{\infty}(M),\Gamma
(\mathcal{E}))$. In particular, $A$ is non-graded. Accordingly, SH LR algebras
generalize $L_{\infty}$ algebroids in two directions: first allowing for more
general algebras $A$ and modules $L$ than algebras of smooth functions and
modules of smooth sections, and second allowing for graded algebras $A$.
\end{example}

Let $(A,L)$ be an $LR_{\infty}[1]$ algebra with \emph{structure
multiderivation} $\mathbb{X}$. The $k$-entry component of $X$ is the
$k$\emph{-th bracket}, and the $k$-entry component of $\sigma_{\mathbb{X}}$ is
the $k$\emph{-th anchor} of $(A,L)$. In terms of brackets and anchors, the
higher Jacobiator $\mathbb{J}(\mathbb{X})=(J(X),\sigma_{\mathbb{J}%
(\mathbb{X})})$ reads
\[
J(X)(\xi_{1},\ldots,\xi_{k})=\sum_{i+j=k}\sum_{\sigma\in S_{i,j}}\alpha
(\sigma,\boldsymbol{\xi})\,X(X(\xi_{\sigma(1)},\ldots,\xi_{\sigma(i)}%
),\xi_{\sigma(i+1)},\ldots,\xi_{\sigma(i+j)}),
\]
and
\begin{align*}
\sigma_{\mathbb{J}(\mathbb{X})}(\xi_{1},\ldots,\xi_{k-1}|a)=  &  \sum
_{i+j=k}\sum_{\sigma\in S_{i,j-1}}(-)^{\chi}\,\alpha(\sigma,\boldsymbol{\xi
})\sigma_{\mathbb{X}}(\xi_{\sigma(1)},\ldots,\xi_{\sigma(i)}|\sigma
_{\mathbb{X}}(\xi_{\sigma(i+1)},\ldots,\xi_{\sigma(i+j-1)}|p))\\
&  +\sum_{i+j=k}\sum_{\sigma\in S_{i,j-1}}\alpha(\sigma,\boldsymbol{\xi
})\,\sigma_{\mathbb{X}}(X(\xi_{\sigma(1)},\ldots,\xi_{\sigma(i)}),\xi
_{\sigma(i+1)},\ldots,\xi_{\sigma(i+j-1)}|p),
\end{align*}
where $\chi=\mathbb{\bar{X}}(\bar{\xi}_{\sigma(1)}+\cdots+\bar{\xi}%
_{\sigma(i)})$, $\boldsymbol{\xi}=(\xi_{1},\ldots,\xi_{k})\in L^{\times k}$,
$a\in A$, and, for simplicity, I omitted the subscript $_{k}$ in the $k$-entry
components of all $K$-multilinear maps. I will adopt the same notation below
when there is no risk of confusion. The above formulas show in particular that
$X$ is an $L_{\infty}[1]$ algebra structure on $L$, and $\sigma_{\mathbb{X}}$
is an $L_{\infty}[1]$ module structure on $A$.

\subsection{SH LR\ algebras, differential algebras, and SH Poisson
algebras\label{SecPoi}}

Let $(A,L)$ be an $LR_{\infty}[1]$ algebra with structure multiderivation
$\mathbb{X} = \mathbb{X}_{1} + \mathbb{X}_{2} + \cdots$. The morphism $\eta$
of Proposition \ref{Prop6} maps $\mathbb{X}$ to a degree $1$ (formal)
derivation $D=D_{1}+D_{2}+\cdots$ of $\mathrm{Sym}_{A}(L,A)$. {Notice that,
since $\mathbb{X}$ has no $0$ entry component, then $D$ has no component
$D_{0}$ mapping $k$-forms to $k$-forms}. Moreover, since $\eta$ preserves the
brackets, $D$ is a homological derivation which amounts to $\sum_{i+j=k}%
[D_{i},D_{j}]=0$, for all $k$. In terms of anchors and brackets $D_{k}$ is
given by the following \emph{higher Chevalley-Eilenberg formula }%
\cite{v12,h13}
\begin{align}
(D_{k}\omega)(\xi_{1},\ldots,\xi_{\ell+k-1}):=  &  \sum_{\sigma\in
S_{k-1,\ell}}(-)^{\chi}\alpha(\sigma,\boldsymbol{\xi})\sigma_{\mathbb{X}}%
(\xi_{\sigma(1)},\ldots,\xi_{\sigma(k-1)}\,|\,\omega(\xi_{\sigma(k)}%
,\ldots,\xi_{\sigma(\ell+k-1)}))\nonumber\\
&  -\sum_{\tau\in S_{k,\ell-1}}(-)^{\omega}\alpha(\tau,\boldsymbol{\xi}%
)\omega(X(\xi_{\tau(1)},\ldots,\xi_{\tau(k)}),\xi_{\tau(k+1)},\ldots,\xi
_{\tau(\ell+k-1)}), \label{CED}%
\end{align}
where $\bar{\chi}=\bar{\omega}(\bar{\xi}_{\sigma(1)}+\cdots+\bar{\xi}%
_{\sigma(k-1)})$, $\omega$ is an $\ell$-form, and $\xi_{1},\ldots,\xi
_{\ell+k-1}\in L$. In view of Proposition \ref{Prop6}, if $L$ is projective
and finitely generated, then an $LR_{\infty}[1]$ algebra structure on $(A,L)$
is equivalent to a formal homological derivation $D$ of $\mathrm{Sym}%
_{A}(L,A)$ such that $D_{0}=0$.

\begin{definition}
The pair $(\mathrm{Sym}_{A}(L,A),D)$ is the \emph{Chevalley-Eilenberg DG
algebra} of $(A,L)$ and it is denoted by $\mathbf{CE}(A,L)$.
\end{definition}

Notice that the projection
\begin{equation}
\mathbf{CE}(A,L)\longrightarrow(A,\sigma_{\mathbb{X}_{1}}) \label{3}%
\end{equation}
is a morphism of DG algebras.

\begin{example}
Let $V$ be an $L_{\infty}[1]$ algebra. Then $(K,V)$ is an $LR_{\infty}[1]$
algebra and $\mathbf{CE}(K,V)$ is nothing but the Chevalley-Eilenberg algebra
of $V$.
\end{example}

\begin{definition}
\label{DefP}A $P_{\infty}[1]$ algebra is an associative, graded commutative,
unital algebra $\mathscr{P}$ equipped with a degree $1$ multiderivation
$\Lambda\in\widehat{\mathrm{Der}}{}^{\bullet}\mathscr{P}$ such that
$\Lambda_{0}=0$ and $J(\Lambda):=\frac{1}{2}[\Lambda,\Lambda]=0$.
\end{definition}

In other words a $P_{\infty}[1]$ algebra structure on $\mathscr{P}$ ($P$ for
\textquotedblleft Poisson\textquotedblright) is an $L_{\infty}[1]$ algebra
structure such that the brackets are multiderivations. Thus, $P_{\infty}[1]$
algebras are homotopy versions of Poisson agebras {(with an additional shift
in degree)}.

\begin{remark}
\label{RemCF}Definition \ref{DefP} differs slightly from an analogous
definition in \cite{cf07}. Namely, in \cite{cf07}, Cattaneo and Felder define
$P_{\infty}$ \emph{algebras} as associative, graded commutative, unital
algebras with an additional $L_{\infty}$ algebra structure (beware,
$L_{\infty}$ not $L_{\infty}[1]$) whose structure maps are graded
skew-symmetric multiderivations.
\end{remark}

\begin{example}
Let $G$ be an associative, graded commutative, unital algebra equipped with a
Gerstenhaber bracket $[{}\cdot{},{}\cdot{}]$. Put $\Lambda_{2}(g,g^{\prime
}):=(-)^{g}[g,g^{\prime}]$. Then $(G,\Lambda=\Lambda_{2})$ is a $P_{\infty
}[1]$ algebra with $\Lambda_{1}=\Lambda_{3}=\cdots=0$.
\end{example}

Let $(A,L)$ be an $LR_{\infty}[1]$ algebra with structure multiderivation
$\mathbb{X}$. The morphism $\nu$ of Proposition \ref{Prop7} maps $\mathbb{X}$
to a degree $1$ derivation $\Lambda=\Lambda_{1}+\Lambda_{2}+\cdots$ of
$S_{A}^{\bullet}L$. Since $\nu$ preserves the brackets, then $[\Lambda
,\Lambda]=0$. In view of Proposition \ref{Prop7}, an $LR_{\infty}[1]$ algebra
structure on $(A,L)$ is actually equivalent to a $P_{\infty}[1]$ algebra
structure $\Lambda$ on $S_{A}^{\bullet}L$ such that $\Lambda_{0}=0$ and
$\Lambda_{k}$ takes $S_{A}^{\ell_{1}}L\times\cdots\times S_{A}^{\ell_{k}}L$ to
$S_{A}^{\ell_{1}+\cdots+\ell_{k}+k-1}L$.

\begin{example}
Let $A$ be the graded algebra of smooth functions on a graded manifold
$\mathcal{M}$ and $L$ be the module of sections of a graded vector bundle
$\mathcal{E}$ over $\mathcal{M}$. Then $L$ is projective and finitely
generated. Moreover, $S_{A}^{\bullet}L$ identifies with the algebra of
fiber-wise polynomial functions on the dual bundle $\mathcal{E}^{\ast}$, and
symmetric forms on $L$ identify with fiber-wise polynomial functions on
$\mathcal{E}$. In their turn, symmetric multiderivations of $S_{A}^{\bullet}L$
identify with (homogenous, fiber-wise polynomial functions) on $T^{\ast
}\mathcal{E}^{\ast}$. Denote by $\{{}\cdot{},{}\cdot{}\}_{\mathcal{E}^{\ast}}$
the canonical Poisson bracket on $C^{\infty}(T^{\ast}\mathcal{E}^{\ast})$.
Finally recall that there is a canonical (Tulczyjew-type \cite{t74})
isomorphism (of double vector bundles over $\mathcal{M}$) $T^{\ast}%
\mathcal{E}^{\ast}\simeq T^{\ast}\mathcal{E}$. An $LR_{\infty}[1]$ algebra
structure in $(A,L)$ is then the same as (see \cite{b11} for the case when
$\mathcal{M}$ is a non-graded manifold, see also the appendix of \cite{op05}),

\begin{enumerate}
\item a degree $1$ function $S$ on $T^{\ast}\mathcal{E}^{\ast}$, such that
$\{{}S{},{}S{}\}_{\mathcal{E}^{\ast}}=0$, $S$ is fiber-wise linear with
respect to projection $T^{\ast}\mathcal{E}^{\ast}\longrightarrow\mathcal{E}$
and vanishes on the graph of the zero section of $T^{\ast}\mathcal{E}^{\ast
}\longrightarrow\mathcal{E}^{\ast}$,

\item a homological vector field on $\mathcal{E}$ tangent to the zero section.
\end{enumerate}

I will name $L_{\infty}[1]$\emph{ algebroid with graded base} any graded
vector bundle $\mathcal{E}$ over a graded base manifold $\mathcal{M}$ with a
homological vector field tangent to the zero section. Let $\mathcal{E}%
\longrightarrow\mathcal{M}$ be an $L_{\infty}[1]$ algebroid with graded base.
Then, in particular, both $\mathcal{E}$ and $\mathcal{M}$ are $Q$-manifolds
(see, for instance, \cite{m06,cs11}), and the zero section is a morphism of
$Q$-manifolds (however, beware that $\mathcal{E}\longrightarrow\mathcal{M}$ is
\emph{not}, in general, a $Q$-bundle in the sense of \cite{ks07}). The
\textquotedblleft transformation $L_{\infty}$ algebroids\textquotedblright\ of
Mehta and Zambon (which are associated to the action of an $L_{\infty}$
algebra on a graded manifold, see Remark 4.5 in \cite{mz12}) are examples of
the $L_{\infty}[1]$ algebroids with graded base defined here. Actually the
former can be generalized to \emph{transformation }$L_{\infty}$
\emph{algebroids} associated to the \emph{action of an }$L_{\infty}[1]$
\emph{algebroid (with graded base)} on a graded fibered manifold (see Example
\ref{ExampALA} in Section \ref{SecAct}).
\end{example}

\begin{example}
[Part I]{\label{ExPLR} Let $P$ be a Poisson algebra and $\Omega^{1}(P)$ the
$P$-module of K\"{a}hler differentials. Recall that $(P,\Omega^{1}(P))$ is a
Lie-Rinehart algebra in a natural way (see, for instance, \cite{h90}).
Similarly, a pair $(\mathscr{P},\Omega^{1}(\mathscr{P}))$, where $\mathscr{P}$
is a $P_{\infty}$ algebra and $\Omega^{1}(\mathscr{P})$ is the $\mathscr{P}$%
-module of \emph{graded K\"{a}hler differentials}, is an $LR_{\infty}[1]$
algebra. Recall that if $A$ is an associative, graded commutative, unital
algebra, the $A$-module of graded K\"{a}hler differentials over $A$ is defined
as the quotient
\[
\Omega^{1}(A):=A\otimes_{K}A[-1]/Q,
\]
where $Q\subset A\otimes_{K}A[-1]$ is the (graded) submodule generated by
elements of the form
\[
1\otimes ab-(-)^{a}a\otimes b-(-)^{(a+1)b}b\otimes a,\quad a,b\in A.
\]
Here $A\otimes_{K}A[-1]$ is equipped with the $A$-module structure inherited
from the first factor. Then, there is a canonical, degree $1$ operator, the
\emph{universal derivation }$d:A\longrightarrow\Omega^{1}(A)$, given by
$a\longmapsto(1\otimes a)+Q$. Clearly, $\Omega^{1}(A)$ is generated by the
image of $d$. Now, let $\mathscr{P}$ be a $P_{\infty}$ algebra (beware
$P_{\infty}$, not $P_{\infty}[1]$) with structure maps $\Lambda_{k}%
:\mathscr{P}^{\times k}\longrightarrow\mathscr{P}$ (see Remark \ref{RemCF}).
Actually, to be consistent with conventions in this paper, I need a slight
modification of Cattaneo-Felder definition. Namely, I assume that the
structure maps $\Lambda_{k}$ have degree $k-2$ (instead of $2-k$). It is easy
to see that there is a unique $LR_{\infty}[1]$ algebra structure $\mathbb{X}$
on $(\mathscr{P},\Omega^{1}(\mathscr{P}))$ such that
\begin{align}
X(df_{1},\ldots,df_{k})  &  =(-)^{\chi}d\Lambda_{k}(f_{1},\ldots
,f_{k})\label{32}\\
\sigma_{\mathbb{X}}(df_{1},\ldots,df_{k-1}|f_{k})  &  =(-)^{\chi^{\prime}%
}\Lambda_{k}(f_{1},\ldots,f_{k-1},f_{k}), \label{33}%
\end{align}
where $\chi=(k-1)\bar{f}_{1}+(k-2)\bar{f}_{2}+\cdots+\bar{f}_{k-1}$,
$\chi^{\prime}=\chi-\bar{f}_{1}-\cdots-\bar{f}_{k-1}$, $f_{1},\ldots,f_{k}%
\in\mathscr{P}$.}

{Notice that, if $\mathscr{P}$ is the algebra of smooth functions on a
\emph{graded manifold }$\mathcal{M}$, then a $P_{\infty}$ algebra structure on
$\mathscr{P}$ is the same as a degree $-2$, skew-symmetric multivector field
$H$ on $\mathcal{M}$, i.e., a degree $2$, fiber-wise polynomial function on
$T^{\ast}[1]\mathcal{M}$, such that $[H,H]_{\mathrm{sn}}=0$, where $[{}\cdot
{},{}\cdot{}]_{\mathrm{sn}}$ is the Schouten-Nijenhuis bracket on graded
multivector fields (see, for instance, \cite{b11}). In this case,
K\"{a}hler differentials on $P$ does not coincide, in general, with
differential $1$-forms on $\mathcal{M}$. However, formulas (\ref{32}) and
(\ref{33}) still define an $LR_{\infty}[1]$ algebra structure on
$(\mathscr{P},\Omega^{1}(\mathcal{M}))$.}
\end{example}

\begin{example}
[Part I]{\label{ExJac} Let }$C${ be a Jacobi algebra, i.e., an associative,
commutative, unital algebra equipped with a Lie algebra structure }$\{{}%
\cdot{},{}\cdot{}\}$ {which is a differential operator of order }$1$, i.e.,
(in this simple case) an $A$-module derivation, in each argument. Moreover,
let{ $J^{1}C$ be the $C$-module of }$1$-jets of $C$ (see, for instance,
\cite{k97}){. The pair $(C,J^{1}(C))$ is a Lie-Rinehart algebra in a natural
way (see, for instance, \cite{v00} for a smooth version of this statement).
}It is natural to define a Jacobi analogue of a $P_{\infty}$ algebra as
follows. A \emph{$J_{\infty}$ algebra} is an associative, graded commutative, unital
algebra $\mathscr{C}$ equipped with an additional $L_\infty$ algebra structure whose structure maps are $A$-module derivations in each argument. {Then a pair $(\mathscr{C},J^{1}\mathscr{C})$, where
$\mathscr{C}$ is a $J_{\infty}$ algebra and $J^{1}\mathscr{C}$ is the
$\mathscr{C}$-module of \emph{graded 1-jets}} of {$\mathscr{C}$, is an
$LR_{\infty}[1]$ algebra. To see this, first recall that if $A$ is an
associative, graded commutative, unital algebra, the $A$-module of }$1$-jets
of ${A}${ is defined as the quotient
\[
J^{1}A:=A\otimes_{K}A[-1]/Q,
\]
where $Q\subset A\otimes_{K}A[-1]$ is the (graded) submodule generated by
elements of the form
\[
1\otimes ab-(-)^{a}a\otimes b-(-)^{(a+1)b}b\otimes a+ab\otimes1,\quad a,b\in
A.
\]
Here $A\otimes_{K}A[-1]$ is equipped with the $A$-module structure inherited
from the first factor. Then, there is a canonical, degree $1$ operator, the
\emph{universal first order differential operator }$j^{1}:A\longrightarrow
J^{1}A$, given by $a\longmapsto(1\otimes a)+Q$. Clearly, $J^{1}A$ is generated
by the image of $j^{1}$. Now, let $\mathscr{C}$ be a $J_{\infty}$ algebra with
structure maps }${\Phi}${$_{k}:\mathscr{C}^{\times k}\longrightarrow
\mathscr{C}$. To be consistent with conventions in this paper, I need the
}$\Phi_{k}$'s to be of degree $k-2${. It is easy to see that there is a unique
$LR_{\infty}[1]$ algebra structure $\mathbb{X}$ on $(\mathscr{C},J^{1}%
\mathscr{C})$ such that
\begin{align}
X(j^{1}f_{1},\ldots,j^{1}f_{k})  &  =(-)^{\chi}j^1\Phi_{k}(f_{1},\ldots,f_{k})\\
\sigma_{\mathbb{X}}(j^{1}f_{1},\ldots,j^{1}f_{k-1}|f_{k})  &  =(-)^{\chi
^{\prime}}(\Phi_{k}(f_{1},\ldots,f_{k-1},f_{k})-\Phi_{k}(f_{1},\ldots
,f_{k-1},1)f_{k}),
\end{align}
where $\chi=(k-1)\bar{f}_{1}+(k-2)\bar{f}_{2}+\cdots+\bar{f}_{k-1}$,
$\chi^{\prime}=\chi-\bar{f}_{1}-\cdots-\bar{f}_{k-1}$, $f_{1},\ldots,f_{k}%
\in\mathscr{C}$.}

Notice that $\mathrm{Sym}_{A}(J^{1}\mathscr{C},\mathscr{C})\simeq
\mathrm{Der}_{\mathscr{C}}^{\bullet}\mathscr{C}$ and the differential induced
by {$\mathbb{X}$} is nothing but $[\Phi,{}\cdot{}]$.
\end{example}

\section{Left SH LR Connections\label{SecLeftConn}}

Left connections along SH LR algebras generalize simultaneously: connections
along LR algebras to the homotopy setting, and representations of $L_{\infty}$
algebras to the LR setting.

Let $(A,L)$ be an $LR_{\infty}[1]$ algebra with structure multiderivation
$\mathbb{X}$, and $P$ an $A$-module.

\begin{definition}
\label{DefLConn}A \emph{left }$(A,L)$\emph{-connection} \emph{in }$P$ is a
degree $1$, $\mathrm{Der}^{L}P$-valued form $\nabla$, such that $(\mathbb{X}%
,\nabla)\in\widehat{\mathcal{L}}(P)$. The $\mathrm{Der}^{L}P$-valued form
$J(\nabla):=\nabla\circ X+\frac{1}{2}[\nabla,\nabla]$ is the \emph{curvature
of }$\nabla$. A left $(A,L)$-connection is \emph{flat} if the curvature
vanishes identically. An $A$-module with a flat left $(A,L)$-connection is a
\emph{left }$(A,L)$\emph{-module}.
\end{definition}

\begin{example}
Let $V$ be an $L_{\infty}[1]$ algebra. Then $(K,V)$ is an $LR_{\infty}[1]$
algebra and left $(K,V)$-modules are just $L_{\infty}[1]$ modules over $V$.
\end{example}

\begin{remark}
\label{RemLCurv}The curvature $J(\nabla)$ of a left $(A,L)$-connection is the
second component of the commutator
\[
\tfrac{1}{2}[(\mathbb{X},\nabla),(\mathbb{X},\nabla)]=(\mathbb{J}%
(\mathbb{X}),J(\nabla)),
\]
whose first component vanishes identically. Accordingly, the symbol of
$J(\nabla)(\xi_{1},\ldots,\xi_{k-1})$ vanishes identically, for all $\xi
_{1},\ldots,\xi_{k-1}\in L$, $k\in\mathbb{N}$, i.e., $J(\nabla)$ takes values
in $\mathrm{End}_{A}P$. Moreover, it follows from the Jacobi identity for the
Lie bracket in $\widehat{\mathcal{L}}(P)$ that $[(\mathbb{X},\nabla
),[(\mathbb{X},\nabla),(\mathbb{X},\nabla)]]=0$, i.e.,
\begin{equation}
\lbrack\nabla,J(\nabla)]-J(\nabla)\circ X=0, \label{LBianchi}%
\end{equation}
which is a higher version of the \emph{Bianchi identity}.
\end{remark}

In terms of the components of $\nabla$, the curvature is given by formulas%
\begin{align}
J(\nabla)(\xi_{1},\ldots,\xi_{k-1}|p):=  &  \sum_{i+j=k}\sum_{\sigma\in
S_{i,j-1}}\alpha(\sigma,\boldsymbol{\xi})\,\nabla(X(\xi_{\sigma(1)},\ldots
,\xi_{\sigma(i)}),\xi_{\sigma(i+1)},\ldots,\xi_{\sigma(i+j-1)}|p)\nonumber\\
&  +\sum_{i+j=k}\sum_{\sigma\in S_{i,j-1}}(-)^{\chi}\,\alpha(\sigma
,\boldsymbol{\xi})\nabla(\xi_{\sigma(1)},\ldots,\xi_{\sigma(i)}|\nabla
(\xi_{\sigma(i+1)},\ldots,\xi_{\sigma(i+j-1)}|p)). \label{LeftCurv}%
\end{align}
where $\chi=\bar{\xi}_{\sigma(1)}+\cdots+\bar{\xi}_{\sigma(i)}$, and
$\boldsymbol{\xi}=(\xi_{1},\ldots,\xi_{k-1})\in L^{\times(k-1)}$, $p\in P$.

\begin{remark}
{Connections along $LR_{\infty}[1]$ algebras as in Definition \ref{DefLConn}
should not be confused with Crainic's \emph{connections up to homotopy}
\cite{c00}. The former are (multi)linear, while the latter are linear only
\textquotedblleft up to homotopy\textquotedblright.}
\end{remark}

The morphism $\eta^{L}$ of Proposition \ref{Prop8} maps $(\mathbb{X},\nabla)$
to a degree $1$, $\mathrm{Sym}_{A}(L,A)$-module derivation $D^{\nabla}%
=D_{1}^{\nabla}+D_{2}^{\nabla}+\cdots$ of $\mathrm{Sym}_{A}(L,P)$ with symbol
$D$. In terms of anchors and brackets, $D_{k}^{\nabla}$ is given by the
following formula%
\begin{align}
(D_{k}^{\nabla}\Omega)(\xi_{1},\ldots,\xi_{\ell+k-1}):=  &  \sum_{\sigma\in
S_{k-1,\ell}}(-)^{\chi^{\prime}}\alpha(\sigma,\boldsymbol{\xi})\nabla
(\xi_{\sigma(1)},\ldots,\xi_{\sigma(k-1)}\,|\,\Omega(\xi_{\sigma(k)}%
,\ldots,\xi_{\sigma(\ell+k-1)}))\nonumber\\
&  -\sum_{\tau\in S_{k,\ell-1}}(-)^{\Omega}\alpha(\tau,\boldsymbol{\xi}%
)\Omega(X(\xi_{\tau(1)},\ldots,\xi_{\tau(k)}),\xi_{\tau(k+1)},\ldots,\xi
_{\tau(\ell+k-1)}), \label{CEDM}%
\end{align}
where $\chi^{\prime}=\bar{\Omega}(\bar{\xi}_{\sigma(1)}+\cdots+\bar{\xi
}_{\sigma(k-1)})$, $\Omega$ is a $P$-valued $\ell$-form, and $\xi_{1}%
,\ldots,\xi_{\ell+k-1}\in L$. In view of Remark \ref{Rem1}, the left
$(A,L)$-connection $\nabla$ in $P$ is actually equivalent to $D^{\nabla}$ and
it is flat iff $\mathcal{J}^{\nabla}:=\frac{1}{2}[D^{\nabla},D^{\nabla}]=0$.
Notice that the symbol of $\mathcal{J}^{\nabla}$ vanishes identically, i.e.,
$\mathcal{J}^{\nabla}$ is a degree $2$, $\mathrm{Sym}_{A}(L,A)$-linear
endomorphism of $\mathrm{Sym}_{A}(L,P)$. In terms of the components of the
curvature, $\mathcal{J}^{\nabla}=\mathcal{J}_{1}^{\nabla}+\mathcal{J}%
_{2}^{\nabla}+\cdots$ is given by formulas
\[
\mathcal{J}_{k}^{\nabla}(\Omega)(\xi_{1},\ldots,\xi_{\ell+k-1}):=\sum
_{\sigma\in S_{k-1,\ell}}(-)^{\chi^{\prime\prime}}\alpha(\sigma;\boldsymbol{\xi
})J(\nabla)(\xi_{\sigma(1)},\ldots,\xi_{\sigma(k-1)}|\Omega(\xi_{\sigma
(k)},\ldots,\xi_{\sigma(\ell+k-1)})),
\]
where $\chi^{\prime\prime}=\bar{\Omega}(\bar{\xi}_{\sigma(1)}+\cdots+\bar{\xi
}_{\sigma(k-1)})$, $\Omega$ is a $P$-valued $\ell$-form, and $\xi_{1}%
,\ldots,\xi_{\ell+k-1}\in L$.

Let $P$ be a left $(A,L)$-module with structure (flat) left connection
$\nabla$.

\begin{definition}
The pair $(\mathrm{Sym}_{A}(L,P),D^{\nabla})$ is the \emph{Chevalley-Eilenberg
DG module} of $P$, and it is denoted by $\mathbf{CE}(P)$.
\end{definition}

Notice that the projection
\[
\mathbf{CE}(P)\longrightarrow(P,\nabla_{1})
\]
is a morphism of DG modules over (\ref{3}).

\begin{example}
\label{ExRepHom}Let $\mathcal{E}$ be a \emph{non-graded} Lie algebroid over a
\emph{non-graded} smooth manifold $M$, and $\mathcal{Y}$ a \emph{graded
}vector bundle over $M$. Denote by $\Omega(\mathcal{E})$ the graded algebra of
sections of the exterior bundle $\Lambda^{\bullet}\mathcal{E}^{\ast}$, and by
$\Omega(\mathcal{E},\mathcal{Y})$ the $\Omega(\mathcal{E})$-module of sections
of the vector bundle $\Lambda^{\bullet}\mathcal{E}^{\ast}\otimes
_{M}\mathcal{Y}$. The pair $(C^{\infty}(M),\Gamma(\mathcal{E}))$ is a
\emph{non-graded} Lie-Rinehart algebra and $(A,L):=(C^{\infty}(M),\Gamma
(E)[1])$ is an $LR_{\infty}[1]$ algebra whose structure multiderivation has
only a two entry component. Representations up to homotopy of $E$ \cite{ac11}
provide examples of left $(A,L)$-modules. Recall that a representation up to
homotopy (in the sense of Definition 3.1 in \cite{ac11}) is a graded vector bundle
$\mathcal{Y}$ equipped with a degree $1$, $\Omega(E)$-module homological
derivation of $\Omega(E,\mathcal{Y})$ subordinate to the Chevalley-Eilenberg
differential in $\Omega(E)$. Clearly, $\Omega(E)\simeq\mathrm{Sym}_{A}(L,A)$
and $\Omega(E,\mathcal{Y})\simeq\mathrm{Sym}_{A}(L,\Gamma(\mathcal{Y}))$.
Moreover $L$ is a projective and finitely generated $A$-module. Therefore, a
representation up to homotopy is equivalent to a left $(A,L)$-module. It turns
out that a good definition of adjoint representation of a Lie algebroid can be
given in the setting of representations up to homotopy. As for Lie algebras,
cohomologies of the adjoint representation control deformations of the Lie
algebroid \cite{cm08}. It would be interesting to explore the problem of
finding a good definition of adjoint representation of an $L_{\infty}$
algebroid (over a possibly graded manifold) or even of adjoint representation
of an SH LR algebra.
\end{example}

\subsection{Operations with Left Connections\label{SecOper}}

Let $(A,L)$ be an $LR_{\infty}[1]$ algebra. Standard connections induce
connections on tensor products and homomorphisms. The same is true for
$(A,L)$-connections. Namely, let $P$ and $P^{\prime}$ be $A$-modules with left
$(A,L)$-connections $\nabla$ and $\nabla^{\prime}$, respectively. It is easy
to see that formulas
\[
\nabla^{\otimes}(\xi_{1},\ldots,\xi_{k-1}|p\otimes p^{\prime}):=\nabla(\xi
_{1},\ldots,\xi_{k-1}|p)\otimes p^{\prime}+(-)^{\chi+p}p\otimes\nabla^{\prime
}(\xi_{1},\ldots,\xi_{k-1}|p^{\prime}),
\]
where $\chi=(\bar{\xi}_{1}+\cdots+\bar{\xi}_{k-1})\bar{p}$, define a left
$(A,L)$-connection $\nabla^{\otimes}$ in $P\otimes_{A}P^{\prime}$. A
straightforward computation shows that the curvature $J(\nabla^{\otimes})$ is
given by formulas
\[
J(\nabla^{\otimes})(\xi_{1},\ldots,\xi_{k-1}|p\otimes p^{\prime}%
)=J(\nabla)(\xi_{1},\ldots,\xi_{k-1}|p)\otimes p^{\prime}+(-)^{\chi}p\otimes
J(\nabla^{\prime})(\xi_{1},\ldots,\xi_{k-1}|p^{\prime}).
\]
In particular, if $\nabla$ and $\nabla^{\prime}$ are flat, then $\nabla
^{\otimes}$ is flat as well.

Similarly, formulas
\[
\nabla^{\mathrm{Hom}}(\xi_{1},\ldots,\xi_{k-1}|\varphi)(p):=\nabla^{\prime
}(\xi_{1},\ldots,\xi_{k-1}|\varphi(p))-(-)^{\chi^{\prime}+\varphi}%
\varphi(\nabla(\xi_{1},\ldots,\xi_{k-1}|p)),
\]
where $\chi^{\prime}=(\bar{\xi}_{1}+\cdots+\bar{\xi}_{k-1})\bar{\varphi}$,
define a left $(A,L)$-connection $\nabla^{\mathrm{Hom}}$ in $\mathrm{Hom}%
_{A}(P,P^{\prime})$. A straightforward computation shows that the curvature
$J(\nabla^{\mathrm{Hom}})$ is given by formulas
\[
J(\nabla^{\mathrm{Hom}})(\xi_{1},\ldots,\xi_{k-1}|\varphi)(p):=J(\nabla
)(\xi_{1},\ldots,\xi_{k-1}|\varphi(p))-(-)^{\chi^{\prime}}\varphi
(J(\nabla^{\prime})(\xi_{1},\ldots,\xi_{k-1}|p)).
\]
In particular, if $\nabla$ and $\nabla^{\prime}$ are flat, then $\nabla
^{\mathrm{Hom}}$ is flat as well.

\begin{remark}
Recall that the connections $\nabla$ and $\nabla^{\prime}$ determine
$\mathrm{Sym}_{A}(L,A)$-module derivations $D^{\nabla}$, and $D^{\nabla
^{\prime}}$ on $\mathrm{Sym}_{A}(L,P)$, and $\mathrm{Sym}_{A}(L,P^{\prime})$
respectively, both with the same symbol $D$. Therefore they also determine a
$\mathrm{Sym}_{A}(L,A)$-module derivation in the tensor product
\[
T:=\mathrm{Sym}_{A}(L,P)\underset{\mathrm{Sym}_{A}(L,A)}{\otimes}%
\mathrm{Sym}_{A}(L,P^{\prime}),
\]
with symbol equal to $D$. If $L$ is projective and finitely generated, then
$T\simeq\mathrm{Sym}_{A}(L,P\otimes P^{\prime})$ and a derivation of $T$ (with
symbol $D$) is the same as a left connection in $P\otimes P^{\prime}$. It is
easy to see that the connection in $P\otimes P^{\prime}$ obtained in this way
is precisely $\nabla^{\otimes}$.
\end{remark}

\begin{remark}
Let $P$ be an $A$-module with a left $(A,L)$-connection $\nabla$. There is an
induced left $(A,L)$-connection $\nabla^{\mathrm{End}}$ in $\mathrm{End}_{A}%
P$. In its turn, $\nabla^{\mathrm{End}}$ determines a derivation
$D^{\nabla^{\mathrm{End}}}$ of $\mathrm{End}_{A}P$-valued forms. On the other
hand, in view of Remark \ref{RemLCurv}, the curvature $J(\nabla)$ of $\nabla$
is an $\mathrm{End}_{A}P$-valued form and
\[
D^{\nabla^{\mathrm{End}}}J(\nabla)=\nabla^{\mathrm{End}}\circ J(\nabla
)-J(\nabla)\circ X=[\nabla,J(\nabla)]-J(\nabla)\circ X=0,
\]
where I used the Bianchi identity \ref{LBianchi}.
\end{remark}

\subsection{Actions of SH LR Algebras\label{SecAct}}

A Lie algebra may act on a manifold. The action is then encoded by a
\emph{transformation }(or, \emph{action}) \emph{Lie algebroid}.\emph{ }More
generally, a Lie algebroid may act on a fibered manifolds over the same base
manifold. The action is again encoded by a transformation Lie algebroid
\cite{hm90}. The algebraic counterpart of a fibered manifold is an algebra
extension. Accordingly, LR algebras may act on algebra extensions, and the
action is encoded by a new LR algebra. On the other hand, $L_{\infty}$
algebras may act on graded manifolds \cite{mz12}. In this section, I show how
to generalize simultaneously actions of Lie algebroids (on fibered manifolds)
and actions of $L_{\infty}$ algebras (on graded manifolds) to \emph{actions of
SH LR algebras on graded algebra extensions}. I also describe the analogue of
the transformation Lie algebroid in this context.

Let $(A,L)$ be an $LR_{\infty}[1]$ algebra with structure multiderivation
$\mathbb{X}$ and $\mathcal{A}$ an associative, graded commutative, unital
$A$-algebra, i.e., a graded algebra extension $A\subset\mathcal{A}$. In
particular, $\mathcal{A}$ is an $A$-module. Of special interest are left
$(A,L)$-connections $\nabla$ in $\mathcal{A}$, which take values in
\emph{derivations }of $\mathcal{A}$ (recall that a general connection takes
values in $A$-module derivations of $\mathcal{A}$, which are more general
operators than algebra derivations).

\begin{definition}
\label{DefAct}A \emph{pre-action} of the $LR_{\infty}[1]$ algebra $(A,L)$ on
$\mathcal{A}$ is a left $(A,L)$-connection $\nabla$ on $\mathcal{A}$ which
takes values in algebra derivations of $\mathcal{A}$, i.e., for all $\xi
_{1},\ldots,\xi_{k-1}\in L$, and $f,g\in\mathcal{A}$
\[
\nabla_{k}(\xi_{1},\ldots,\xi_{k-1}|{}fg{})=\nabla_{k}(\xi_{1},\ldots
,\xi_{k-1}|{}f{})\,g+(-)^{gf}\nabla_{k}(\xi_{1},\ldots,\xi_{k-1}|{}g{})\,f.
\]
A flat pre-action is an \emph{action}.
\end{definition}

\begin{remark}
An associative, graded commutative, unital $A$-algebra with an action $\nabla$
of $(A,L)$ is, in particular, a DG algebra (just forget about all the
structure maps $\nabla_{k}$ except the first one).
\end{remark}

\begin{example}
Let $\nabla$ be a left $(A,L)$-connection in a left $A$-module $P$. The
induced $(A,L)$-connections in $S_{A}^{\bullet}P$ and $\mathrm{Sym}_{A}(P,A)$
are obviously pre-actions, and they are actions if $\nabla$ is flat (see
Section \ref{SecOper}). Pre-actions on $S_{A}^{\bullet}P$ induced by left
$(A,L)$-connections in $P$ are characterized by the condition that the
derivations $\nabla(\xi_{1},\ldots,\xi_{k-1})$ restrict to $P$. Similarly, if
$P$ is a projective and finitely generated $A$-module, pre-actions on
$\mathrm{Sym}_{A}(P,A)$ induced by left $(A,L)$-connections in $P$ are
characterized by the condition that the derivations $\nabla(\xi_{1},\ldots
,\xi_{k-1})$ restrict to $P^{\ast}:=\mathrm{Hom}_{A}(P,A)$. Indeed, under the
regularity hypothesis on $P$, one has $\mathrm{Sym}_{A}(P,A)\simeq
S_{A}^{\bullet}P^{\ast}$ and, since $P\simeq P^{\ast\ast}$, left
$(A,L)$-connections in $P$ are equivalent to left $(A,L)$-connections in
$P^{\ast}$.
\end{example}

{The following example shows that the notion of left connection along a SH LR
algebra captures both \emph{representations up to homotopy of Lie algebroids}
(see Example \ref{ExRepHom}) and \emph{actions of $L_{\infty}$ algebras on
graded manifolds} in the sense of \cite{mz12}.}

\begin{example}
Let $A$ be the graded algebra of smooth functions on a graded manifold
$\mathcal{M}$ and $L$ be the module of sections of a graded vector bundle
$\mathcal{E}$ over $\mathcal{M}$. Moreover, let $\mathcal{A}$ be the
$A$-algebra of smooth functions on a graded manifold $\mathcal{F}$ which is
fibered over $\mathcal{M}$. As already remarked, an $LR_{\infty}[1]$ algebra
structure in $L$ is the same as a homological vector field $D$ on
$\mathcal{E}$ tangent to the zero section. Moreover, the $\mathrm{Sym}%
_{A}(L,A)$-algebra $\mathrm{Sym}_{A}(L,\mathcal{A})$ identifies with the
module of fiber-wise polynomial functions on the total space $\mathcal{E}%
\times_{\mathcal{M}}\mathcal{F}$ of the induced bundle $\pi:\mathcal{E}%
\times_{\mathcal{M}}\mathcal{F}\longrightarrow{}\mathcal{E}$. In their turn,
derivations of $\mathrm{Sym}_{A}(L,\mathcal{A})$ identify with vector fields
on $\mathcal{E}\times_{\mathcal{M}}\mathcal{F}$. Therefore, a pre-action
$\nabla$ of $(A,L)$ on $\mathcal{A}$ is the same as a vector field $D^{\nabla
}$ on $\mathcal{E}\times_{\mathcal{M}}\mathcal{F}$ which is $\pi$-related to
$D$. Moreover, $\nabla$ is an action iff $D^{\nabla}$ is a homological vector
field. Now, let $\mathcal{M}=\{\ast\}$ be a point. Then $A=\mathbb{R}$ and $L$
is just an $L_{\infty}[1]$ algebra. Moreover, $\mathcal{E}$ is just a graded
manifold and the notion of action of $(A,L)$ on $\mathcal{A}$ reduces to the
notion of action of an $L_{\infty}[1]$ algebra on a graded manifold proposed
by Mehta and Zambon \cite{mz12}.

A special situation is when $\mathcal{F}$ is a vector bundle over
$\mathcal{M}$. In this case, $\mathcal{A}=S_{A}^{\bullet}\Gamma(\mathcal{F}%
)^{\ast}$ (up to \emph{smooth completion}), and an $(A,L)$-connection in
$\Gamma(\mathcal{F})$ determines a pre-action of $(A,L)$ on $\mathcal{A}$. It
is easy to see that pre-actions on $\mathcal{A}$ of this form are
characterized by the condition that the vector field $D^{\nabla}$ is
fiber-wise linear with respect to the vector bundle projection $\pi
:\mathcal{E}\times_{\mathcal{M}}\mathcal{F}\longrightarrow{}\mathcal{E}$.
\end{example}

\begin{proposition}
\label{Prop1}The following data are equivalent:

\begin{enumerate}
\item a pre-action of $(A,L)$ on $\mathcal{A}$,

\item a degree $1$ derivation of the algebra $\mathrm{Sym}_{A}(L,\mathcal{A})$
which agrees with $D$ on $\mathrm{Sym}_{A}(L,A)$,

\item an $\mathcal{A}$-module multiderivation of $\mathcal{A}\otimes_{A}L$
which agrees with $\mathbb{X}$ on $(A,L)$.
\end{enumerate}
\end{proposition}

\begin{proof}
It is straightforward to check that (1) $\Longleftrightarrow$ (2): the
derivation of $\mathrm{Sym}_{A}(L,\mathcal{A})$ corresponding to a pre-action
$\nabla$ is just $D^{\nabla}$. Now, notice that an $\mathcal{A}$-module
multiderivation $\mathbb{Y}=(Y,\sigma_{\mathbb{Y}})$ of $\mathcal{A}%
\otimes_{A}L$ is completely determined by restrictions of $Y$ and
$\sigma_{\mathbb{Y}}$ to generators, i.e., elements of $L$. If $\mathbb{Y}$
agrees with $\mathbb{X}$ on $(A,L)$ then the restrictions $\sigma_{\mathbb{Y}%
}:L\times\cdots\times L\longrightarrow\mathrm{Der}\mathcal{A}$ of the symbol
determine a pre-action of $(A,L)$ on $\mathcal{A}$. Conversely, the structure
maps $\nabla:L\times\cdots\times L\longrightarrow\mathrm{Der}\mathcal{A}$ of a
pre-action can be extended to maps
\[
\sigma:(\mathcal{A}\otimes L)\times\cdots\times(\mathcal{A}\otimes
L)\longrightarrow\mathrm{Der}\mathcal{A}%
\]
by $\mathcal{A}$-linearity. Hence, $\mathbb{X}$ can be extended to an
$\mathcal{A}$-module multiderivation $\mathbb{Y}$ of $\mathcal{A}\otimes_{A}L$
demanding that $\sigma$ is its symbol.
\end{proof}

\begin{corollary}
\label{Cor1}The following data are equivalent

\begin{enumerate}
\item an action of $(A,L)$ on $\mathcal{A}$,

\item a degree $1$ homological derivation of the algebra $\mathrm{Sym}%
_{A}(L,\mathcal{A})$ which agrees with $D$ on $\mathrm{Sym}_{A}(L,A)$,

\item an $LR_{\infty}[1]$ algebra structure on $(\mathcal{A},\mathcal{A}%
\otimes_{A}L)$ which agrees with $\mathbb{X}$ on $(A,L)$.
\end{enumerate}
\end{corollary}

Let $\nabla$ be an action of $(A,L)$ on $\mathcal{A}$. In view of Proposition
\ref{Prop1}, the Chevalley-Eilenberg DG module $\mathbf{CE}(\mathcal{A})$ of
$\mathcal{A}$ is actually a DG algebra. On the other hand, the
Chevalley-Eilenberg DG algebra of $(\mathcal{A},\mathcal{A}\otimes_{A}L)$ is
\begin{equation}
\mathbf{CE}(\mathcal{A},\mathcal{A}\otimes_{A}L)=(\mathrm{Sym}_{\mathcal{A}%
}(\mathcal{A}\otimes_{A}L,\mathcal{A}),D)\simeq(\mathrm{Sym}_{A}%
(L,\mathcal{A}),D^{\nabla})=\mathbf{CE}(\mathcal{A}), \label{7}%
\end{equation}
i.e., it is canonically isomorphic to the Chevalley-Eilenberg DG module of
$\mathcal{A}$. In particular, there is an obvious sequence of DG algebras (and
$A$-modules)
\begin{equation}
\mathbf{CE}(A,L)\longrightarrow\mathbf{CE}(\mathcal{A})\longrightarrow
(\mathcal{A},\nabla_{1}). \label{8}%
\end{equation}

\begin{remark}
The $\mathcal{A}$-algebra isomorphism $\mathrm{Sym}_{\mathcal{A}}%
(\mathcal{A}\otimes_{A}L,\mathcal{A})\simeq\mathrm{Sym}_{A}(L,\mathcal{A})$
provides an alternative proof of the equivalence of data (2) and (3) in
Proposition \ref{Prop1} in the case when $L$ is projective and finitely
generated. Indeed, in this case, $\mathcal{A}\otimes_{A}L$ is a projective and
finitely generated $\mathcal{A}$-module and a derivation of $\mathrm{Sym}%
_{\mathcal{A}}(\mathcal{A}\otimes_{A}L,\mathcal{A})$ is equivalent to an
$\mathcal{A}$-module multiderivation of $\mathcal{A}\otimes_{A}L$.
\end{remark}

In view of the following example, it is natural to call \emph{transformation
}$LR_{\infty}[1]$\emph{ algebra} the $LR_{\infty}[1]$ algebra $(\mathcal{A}%
,\mathcal{A}\otimes_{A}L)$ corresponding to an action of $(A,L)$ on
$\mathcal{A}$ via Corollary \ref{Cor1}.

\begin{example}
\label{ExampALA}Let $A$ be the graded algebra of smooth functions on a graded
manifold $\mathcal{M}$ and $L$ be the module of sections of a graded vector
bundle $\mathcal{E}$ over $\mathcal{M}$. Moreover, let $\mathcal{A}$ be the
$A$-algebra of smooth functions on a graded manifold $\mathcal{F}$ which is
fibered over $\mathcal{M}$, and let $\nabla$ be an action of $(A,L)$ on
$\mathcal{A}$. Notice that the $\mathcal{A}$-module $\mathcal{A}\otimes_{A}L$
is the module of sections of the induced bundle $\xi:\mathcal{E}%
\times_{\mathcal{M}}\mathcal{F}\longrightarrow$ $\mathcal{F}$. It is easy to
see that the vector field $D^{\nabla}$ is tangent to the zero section of $\xi
$. This shows that $\mathcal{E}\times_{\mathcal{M}}\mathcal{F}$ has the
structure of an $L_{\infty}[1]$ algebroid over $\mathcal{F}$. I call any such
$L_{\infty}[1]$ algebroid a \emph{transformation} $L_{\infty}[1]$
\emph{algebroid }and denote it by $\mathcal{E}\ltimes_{\mathcal{M}}%
\mathcal{F}$. The transformation $L_{\infty}[1]$ algebroid $\mathcal{E}%
\ltimes_{\mathcal{M}}\mathcal{F}$ fits into the sequence of $Q$-manifolds (and
bundles over $\mathcal{M}$)
\begin{equation}
(\mathcal{F},\nabla_{1})\longrightarrow(\mathcal{E}\ltimes_{\mathcal{M}%
}\mathcal{F},D^{\nabla})\longrightarrow(\mathcal{E},D), \label{2}%
\end{equation}
where the first arrow is the zero section. Sequence (\ref{2}) is the geometric
counterpart of Sequence (\ref{8}) and generalizes Sequence (7) of \cite{mz12}.
\end{example}

\subsection{{Higher Left Schouten-Nijenhuis Calculus\label{SecLSN}}}

Cartan calculus on a smooth manifold is the calculus of vector fields and
differential forms. Schouten-Nijenhuis calculus is the calculus of
(skew-symmetric) multivector fields and differential forms. The latter
consists of some identities involving the Schouten-Nijenhuis bracket, the
insertion of multivectors into differential forms, and the exterior
derivative. Namely, let $i_{u}$ be the insertion of a multivector $u$ into
differential forms. The \emph{Lie derivative }along $u$ is the operator
$L_{u}:=[i_{u},d]$, where $d$ is the exterior derivative. Denote by
$[u,v]_{\mathrm{sn}}$ the Schouten-Nijenhuis bracket of multivectors $u,v$.
The following identity holds%

\begin{equation}
\lbrack\lbrack d,i_{u}],i_{v}]=-(-)^{u}i_{[u,v]_{\mathrm{sn}}}, \label{31}%
\end{equation}
$u,v$ multivectors. Moreover, from (\ref{31}), and the definition of the Lie
derivative, it follows immediately that
\begin{align}
\lbrack L_{u},L_{v}]  &  =L_{[u,v]_{\mathrm{sn}}},\label{Schout}\\
L_{u}i_{v}  &  =i_{[u,v]_{\mathrm{sn}}}-(-)^{u}i_{v}L_{u},\label{Schout2}\\
L_{uv}  &  =i_{u}L_{v}+(-)^{v}L_{u}i_{v}. \label{Schout3}%
\end{align}
The above identities can be generalized as follows. Replace ordinary
differential forms with differential forms with values in a vector bundle
equipped with a connection. Replace $d$ with the Chevalley-Eilenberg operator
of the connection. All the above identities remain valid except (\ref{Schout})
which gains terms on the right hand side depending on the curvature of the
connection. Finally, Schouten-Nijenhuis calculus can be easily extended to Lie
algebroids (actually, LR algebras) more general than the tangent bundle. In
this section, I generalize further to SH LR algebras and left connections
along them. Higher derived brackets \cite{v05,u11} play here a key role.

Let $(A,L)$ be an $LR_{\infty}[1]$ algebra with structure multiderivation
$\mathbb{X}$. Recall that $\mathbb{X}$ determines, via Proposition
\ref{Prop7}, a $P_{\infty}[1]$ algebra structure $\Lambda$ on $S_{A}^{\bullet
}L$, i.e., a multiderivation $\Lambda=\Lambda_{1}+\Lambda_{2}+\cdots$ such
that $[\Lambda,\Lambda]=0$. In particular, $\Lambda$ is an $L_{\infty}[1]$
algebra structure on $S_{A}^{\bullet}L$. In other words, $(K,S_{A}^{\bullet
}L)$ is an $LR_{\infty}[1]$ algebra. In this section I will adopt this
interpretation. In the following denote
\[
\{u_{1},\ldots,u_{k}\}:=\Lambda(u_{1},\ldots,u_{k}),\quad u_{1},\ldots
,u_{k}\in S_{A}^{\bullet}L,\quad k\in\mathbb{N}.
\]
Consider an $A$-module $P$ with a left $(A,L)$-connection $\nabla$, and the
corresponding $\mathrm{Sym}_{A}(L,A)$-module derivation $D^{\nabla}%
=D_{1}^{\nabla}+D_{2}^{\nabla}+\cdots$ of $\mathrm{Sym}_{A}(L,P)$.

Recall that, if $B$ is an associative, graded commutative, unital algebra and
$R$ is a $B$-module, a {(linear)} \emph{differential operator of order }$k$ in
$R$ is a $K$-linear map $\square:R\longrightarrow R$ such that
\[
\lbrack\cdots\lbrack\lbrack\square,b_{1}],b_{2}]\cdots,b_{k+1}]=0,\quad
b_{1},\ldots,b_{k+1}\in P,
\]
where the $b_{i}$'s are interpreted as multiplication operators
$P\longrightarrow P$, $p\longmapsto b_{i}p$ {(see, for instance, \cite{k97})}.

\begin{proposition}
\label{Prop2}{For all $u_{1},\ldots,u_{k}\in S_{A}^{\bullet}L$,
\begin{equation}
\lbrack\cdots\lbrack\lbrack D_{k}^{\nabla},i_{u_{1}}],i_{u_{2}}]\cdots
,i_{u_{k}}]=-(-)^{k}i_{\{u_{1},\ldots,u_{k}\}}. \label{5}%
\end{equation}
In particular, $D_{k}^{\nabla}$ is a differential operator of order $k$ in the
$S_{A}^{\bullet}L$-module $\mathrm{Sym}_{A}(L,P)$.}
\end{proposition}

\begin{proof}
{First of all, notice that, since $[i_{u},i_{v}]=0$ for all tensors $u,v$,
then (\ref{5}) implies that $D_{k}^{\nabla}$ is a differential operator of
order $k$, as claimed. Now I prove (\ref{5}).} Suppose, preliminarly, that
identity (\ref{5}) is true when $u_{1},\ldots,u_{k}$ are either from $A$ or
from $L$. Now, let $u_{i}\in S_{A}^{\ell_{i}}L$, $i=1,\ldots,k$. Since
$[i_{uv},\mathcal{D}]=i_{u}[i_{v},\mathcal{D}]+(-)^{v\mathcal{D}}%
[i_{u},\mathcal{D}]i_{v}$, for all symmetric tensors $u,v$ and all graded
$K$-linear endomorphisms $\mathcal{D}$ of $\mathrm{Sym}_{A}(L,P)$, the general
assertion (\ref{5}) follows from an easy induction on the $\ell_{i}$. It
remains to prove (\ref{5}) when $u_{1},\ldots,u_{k}$ are either $0$ or $1$-tensors.

Put $\mathcal{I}(u_{1},\ldots,u_{k}):=[\cdots\lbrack\lbrack D_{k}^{\nabla
},i_{u_{1}}],i_{u_{2}}]\cdots,i_{u_{k}}]$. Since $[i_{u},i_{v}]=0$ for all
tensors $u,v$, then, in view of the Jacobi identity for the graded commutator,
$\mathcal{I}$ is graded symmetric in its arguments. Moreover, it vanishes
whenever two arguments are from $A$. Now, let $a\in A$, and let $\xi
_{1},\ldots,\xi_{k-1}\in L$. It is easy to see that
\[
\mathcal{I}(a,\xi_{1},\ldots,\xi_{k-1}){}=-(-)^{k}i_{\{a,\xi_{1},\ldots
,\xi_{k-1}\}}.
\]
Finally, I have to show that
\[
\mathcal{I}(\xi_{1},\ldots,\xi_{k}){}=-(-)^{k}i_{\{\xi_{1},\ldots,\xi_{k}%
\}},\quad\xi_{1},\ldots,\xi_{k}\in L.
\]
Notice that $\mathcal{I}(\xi_{1},\ldots,\xi_{k})$ is an $A$-linear
endomorphism of $\mathrm{Sym}_{A}(L,P)$, and, for $\Omega$ an $r$-form,
$\mathcal{I}(\xi_{1},\ldots,\xi_{k})(\Omega)$ is an $(r-1)$-form. Thus, I have
to show that
\begin{equation}
\mathcal{I}(\xi_{1},\ldots,\xi_{k})(\Omega)(\zeta_{1},\ldots,\zeta
_{r-1})=-(-)^{k}(i_{\{\xi_{1},\ldots,\xi_{k}\}}\Omega)(\zeta_{1},\ldots
,\zeta_{r-1}), \label{14}%
\end{equation}
for all $\Omega$ and all $\zeta_{1},\ldots,\zeta_{r-1}\in L$. For $r=0$ both
sides of (\ref{14}) vanish. Now, let $r>0$. Since both sides of (\ref{14}) are
graded symmetric in $\xi_{1},\ldots,\xi_{k}$ on one side, and in $\zeta
_{1},\ldots,\zeta_{r-1}$ on the other side, it is enough to consider the case
when $\xi_{1}=\cdots=\xi_{k}=\xi$, $\zeta_{1}=\cdots=\zeta_{r-1}=\zeta$, and
$\xi$ and $\zeta$ are even. Thus, put
\[
\xi^{m}:=(\underset{m\text{ times}}{\underbrace{\xi,\ldots,\xi}})\in L^{\times
m},
\]
and similarly for $\zeta$. Compute
\begin{equation}
\mathcal{I}(\xi^{k})(\Omega)(\zeta^{r-1})=\sum_{j=0}^{k}(-)^{k-j}\tbinom{k}%
{j}(i_{\xi}^{k-j}D_{k}^{\nabla}i_{\xi}^{j}\Omega)(\zeta^{r-1}), \label{20}%
\end{equation}
distinguishing the following two cases.

\textbf{Case I :} $k\leq r$. A straightforward computation shows that the
right hand side of (\ref{20}) is%
\begin{align*}
&  \sum_{i=1}^{k}\sum_{j=1}^{i}(-)^{k-j}\tbinom{k}{j}\tbinom{k-j}{k-i}%
\tbinom{r-1}{i-1}\nabla_{k}(\xi^{k-i},\zeta^{i-1}|\Omega(\xi^{i},\zeta
^{r-k+i}))\\
&  -\sum_{i=1}^{k}\sum_{j=0}^{i}(-)^{\Omega+k-j}\tbinom{k}{j}\tbinom{k-j}%
{k-i}\tbinom{r-1}{i}\Omega(\{\xi^{k-i},\zeta^{i}\},\xi^{i},\zeta^{r-i-1})\\
&  +\sum_{i=1}^{k}(-)^{k}\tbinom{k}{i}\tbinom{r-1}{i-1}\nabla_{k}(\xi
^{k-i},\zeta^{i-1}|\Omega(\xi^{i},\zeta^{r-i}))-(-)^{\Omega}\Omega(\{\xi
^{k}\},\zeta^{r-1})
\end{align*}
Now,%
\begin{equation}
\sum_{j=\varepsilon}^{i}(-)^{j}\tbinom{k}{j}\tbinom{k-j}{k-i}=\left\{
\begin{array}
[c]{cc}%
0 & \text{if }\varepsilon=0\\
-\tbinom{k}{i} & \text{if }\varepsilon=1
\end{array}
\right.  . \label{21}%
\end{equation}
Substituting above, one gets%
\[
\mathcal{I}(\xi^{k})(\Omega)(\zeta^{r-1})=-(-)^{k}(i_{\{\xi^{k}\}}%
\Omega)(\zeta^{r-1}).
\]
\textbf{Case II: }$k>r$. The right hand side of (\ref{20}) is
\begin{align*}
&  \sum_{i=1}^{r}\sum_{j=1}^{i}(-)^{k-j}\tbinom{k}{i}\tbinom{i}{j}\tbinom
{r-1}{i-1}\nabla_{k}(\xi^{k-i},\zeta^{i-1}|\Omega(\xi^{i},\zeta^{r-k+i}))\\
&  +\sum_{i=r+1}^{k}\sum_{j=1}^{r}(-)^{k-j}\tbinom{k}{i}\tbinom{i}{j}%
\tbinom{r-1}{i-1}\nabla_{k}(\xi^{k-i},\zeta^{i-1}|\Omega(\xi^{i},\zeta
^{r-k+i}))\\
&  -\sum_{i=0}^{r}\sum_{j=0}^{i}(-)^{\Omega+k-j}\tbinom{k}{i}\tbinom{i}%
{j}\tbinom{r-1}{i}\Omega(\{\xi^{k-i},\zeta^{i}\},\xi^{i},\zeta^{r-i-1})\\
&  -\sum_{i=r+1}^{k}\sum_{j=0}^{r}(-)^{\Omega+k-j}\tbinom{k}{i}\tbinom{i}%
{j}\tbinom{r-1}{i}\Omega(\{\xi^{k-i},\zeta^{i}\},\xi^{i},\zeta^{r-i-1})\\
&  +\sum_{i=1}^{k}(-)^{k}\tbinom{k}{i}\tbinom{r-1}{i-1}\nabla_{k}(\xi
^{k-i},\zeta^{i-1}|\Omega(\xi^{i},\zeta^{r-i})),
\end{align*}
and, using again (\ref{21}), one gets%
\begin{align*}
\mathcal{I}(\xi^{k})(\Omega)(\zeta^{r-1})  &  =-\sum_{i=1}^{r}(-)^{k}%
\tbinom{k}{i}\tbinom{r-1}{i-1}\nabla_{k}(\xi^{k-i},\zeta^{i-1}|\Omega(\xi
^{i},\zeta^{r-k+i}))\\
&  -\sum_{i=r+1}^{k}(-)^{k}\tbinom{k}{i}\tbinom{r-1}{i-1}\nabla_{k}(\xi
^{k-i},\zeta^{i-1}|\Omega(\xi^{i},\zeta^{r-k+i}))\\
&  +\sum_{i=1}^{k}(-)^{k}\tbinom{k}{i}\tbinom{r-1}{i-1}\nabla_{k}(\xi
^{k-i},\zeta^{i-1}|\Omega(\xi^{i},\zeta^{r-i}))\\
&  -(-)^{\Omega+k}\Omega(\{\xi^{k}\},\zeta^{r-1})=-(-)^{k}(i_{\{\xi^{k}%
\}}\Omega)(\zeta^{r-1}).
\end{align*}
\end{proof}

Identity (\ref{5}) is a homotopy version of identity (\ref{31}). Now I define
a \textquotedblleft\emph{homotopy Lie derivative}\textquotedblright\ and prove
homotopy versions of identities (\ref{Schout}), (\ref{Schout2}), and
(\ref{Schout3}). For $u_{1},\ldots,u_{k-1}\in S_{A}^{\bullet}L$, and
$\Omega\in\mathrm{Sym}_{A}(L,P)$, put
\[
L_{k}^{\nabla}(u_{1},\ldots,u_{k-1}|\Omega):=-(-)^{k}[\cdots\lbrack\lbrack
D_{k}^{\nabla},i_{u_{1}}],i_{u_{2}}]\cdots,i_{u_{k-1}}]\Omega.
\]

\begin{theorem}
\label{Theor1}The sum $L^{\nabla}:=L_{1}^{\nabla}+L_{2}^{\nabla}+\cdots$ is a
$(K,S_{A}^{\bullet}L)$-connection in $\mathrm{Sym}_{A}(L,P)$ whose curvature
$J(L^{\nabla})$ is given by
\begin{equation}
J(L^{\nabla})(u_{1},\ldots,u_{k-1}|\Omega)=-(-)^{k}[\cdots\lbrack
\lbrack\mathcal{J}_{k}^{\nabla},i_{u_{1}}],i_{u_{2}}]\cdots,i_{u_{k-1}}%
]\Omega. \label{22}%
\end{equation}
Moreover
\begin{equation}
L^{\nabla}(u_{1},\ldots,u_{k-1}|i_{u}\Omega)=i_{\{u_{1},\ldots,u_{k-1}%
,u\}}\Omega+(-)^{\chi}i_{u}L^{\nabla}(u_{1},\ldots,u_{k-1}|\Omega) \label{23}%
\end{equation}
where $\chi=\bar{u}(\bar{u}_{1}+\cdots+\bar{u}_{k-1}+1)$, and
\begin{equation}
L^{\nabla}(uu_{1},u_{2},\ldots,u_{k-1}|\Omega)=(-)^{u}i_{u}L^{\nabla}%
(u_{1},\ldots,u_{k-1}|\Omega)+(-)^{\chi^{\prime}}L^{\nabla}(u,u_{2}%
,\ldots,u_{k-1}|i_{u_{1}}\Omega) \label{24}%
\end{equation}
where $\chi^{\prime}=\bar{u}_{1}(\bar{u}_{2}+\cdots+\bar{u}_{k-1})$, for all
tensors $u,u_{1},\ldots,u_{k}$, and all $P$-valued forms $\Omega$.
\end{theorem}

\begin{proof}
Formula (\ref{22}) is a straightforward consequence of Lemma 4.2 of
\cite{u11}. Formula (\ref{23}) immediately follows from Proposition
\ref{Prop2}. Formula (\ref{24}) is a consequence of the Leibniz formula for
the graded commutator.
\end{proof}

\begin{corollary}
Let $\nabla$ be a left $(A,L)$-connection in $P$. If $\nabla$ is flat, then
$L^{\nabla}$ equips $\mathrm{Sym}_{A}(L,P)$ with the structure of a left
$L_{\infty}[1]$ module over the $L_{\infty}[1]$ algebra $S_{A}^{\bullet}L$.
If, in addition, $L\ $is projective and finitely generated, then the converse
is also true.
\end{corollary}

\section{Right SH LR Connections\label{SecRightConn}}

Let $(A,L)$ be an $LR_{\infty}[1]$ algebra with structure multiderivation
$\mathbb{X}$, and $Q$ an $A$-module.

\begin{definition}
\label{DefRConn}A \emph{right }$(A,L)$\emph{-connection} in $Q$ is a degree
$1$, $\mathrm{Der}^{R}P$-valued form $\Delta$, such that $(\mathbb{X}%
,\Delta)\in\widehat{\mathcal{R}}(P)$. The $\mathrm{Der}^{R}P$-valued form
$J(\Delta):=\Delta\circ X-\frac{1}{2}[\Delta,\Delta]$ is the \emph{curvature
of }$\Delta$. A right $(A,L)$-connection is \emph{flat} if the curvature
vanishes identically. An $A$-module with a flat right $(A,L)$-connection is a
\emph{right }$(A,L)$\emph{-module}.
\end{definition}

\begin{remark}
\label{RemRCurv}The curvature $J(\Delta)$ of a right $(A,L)$-connection is the
second component of the commutator
\[
\tfrac{1}{2}[(\mathbb{X},\Delta),(\mathbb{X},\Delta)]=(\mathbb{J}%
(\mathbb{X}),J(\Delta))
\]
whose first entry vanishes identically. Accordingly, the symbol of
$J(\Delta)(\xi_{1},\ldots,\xi_{k-1})$ vanishes identically, for all $\xi
_{1},\ldots,\xi_{k-1}\in L$, $k\in\mathbb{N}$, i.e., $J(\Delta)$ takes values
in $\mathrm{End}_{A}P$. Moreover, it follows from the Jacobi identity for the
Lie bracket in $\widehat{\mathcal{R}}(P)$ that $[(\mathbb{X},\Delta
),[(\mathbb{X},\Delta),(\mathbb{X},\Delta)]]=0$, i.e.,
\begin{equation}
\lbrack\Delta,J(\Delta)]+J(\Delta)\circ X=0, \label{RBianchi}%
\end{equation}
which is a \emph{higher right Bianchi identity}.
\end{remark}

In terms of the components of $\Delta$, the curvature is given by formulas
\begin{align}
J(\Delta)_{k}(\xi_{1},\ldots,\xi_{k-1}|p):=  &  \sum_{i+j=k}\sum_{\sigma\in
S_{i,j-1}}\alpha(\sigma,\boldsymbol{\xi})\,\Delta(X(\xi_{\sigma(1)},\ldots
,\xi_{\sigma(i)}),\xi_{\sigma(i+1)},\ldots,\xi_{\sigma(i+j-1)}|q)\nonumber\\
&  -\sum_{i+j=k}\sum_{\sigma\in S_{i,j-1}}(-)^{\chi}\,\alpha(\sigma
,\boldsymbol{\xi})\Delta(\xi_{\sigma(1)},\ldots,\xi_{\sigma(i)}|\Delta
(\xi_{\sigma(i+1)},\ldots,\xi_{\sigma(i+j-1)}|q)), \label{RightCurv}%
\end{align}
where $\chi=\bar{\xi}_{\sigma(1)}+\cdots+\bar{\xi}_{\sigma(i)}$, and
$\boldsymbol{\xi}=(\xi_{1},\ldots,\xi_{k-1})\in L^{\times(k-1)}$, $q\in Q$.
Notice the minus sign in front of the second summand of the right hand side of
(\ref{RightCurv}), in contrast with Formula (\ref{LeftCurv}).

The morphism $\eta^{R}$ of Proposition \ref{Prop9} maps $(\mathbb{X},\Delta)$
to a degree $1$, $\mathrm{Sym}_{A}(L,A)$-module derivation $D^{\Delta}%
=D_{1}^{\Delta}+D_{2}^{\Delta}+\cdots$ of $S_{A}^{\bullet}L\otimes_{A}Q$ with
symbol $D$. In terms of anchors and brackets, $D_{k}^{\Delta}$ is given by the
following formula
\begin{align}
D_{k}^{\Delta}(\xi_{1}\cdots\xi_{\ell}\otimes q):=  &  \sum_{\tau\in
S_{k, \ell - k}}\alpha(\tau,\boldsymbol{\xi})X(\xi_{\tau(1)},\ldots,\xi_{\tau
(k)})\xi_{\tau(k+1)}\cdots\xi_{\tau(\ell)}\otimes q\nonumber\\
&  -\sum_{\sigma\in S_{\ell-k +1,k-1}}(-)^{\chi}\alpha(\sigma,\boldsymbol{\xi
})\xi_{\sigma(1)}\cdots\xi_{\sigma(\ell-k+1)}\otimes\Delta(\xi_{\sigma
(\ell-k+2)},\ldots,\xi_{\sigma(\ell)}|q) \label{Rine}%
\end{align}
where $\chi=\bar{\xi}_{\sigma(1)}+\cdots+\bar{\xi}_{\sigma(\ell-k+1)}$,
$\xi_{1},\ldots,\xi_{\ell}\in L$, and $q\in Q$. In view of Remark \ref{Rem1},
the right $(A,L)$-connection $\Delta$ is actually equivalent to $D^{\Delta}$
and it is flat iff $\mathcal{J}^{\Delta}:=\frac{1}{2}[D^{\Delta},D^{\Delta
}]=0$. Notice that the symbol of $\mathcal{J}^{\Delta}$ vanishes identically,
i.e., $\mathcal{J}^{\Delta}$ is a degree $2$, $\mathrm{Sym}_{A}(L,A)$-linear
endomorphism of $S_{A}^{\bullet}L\otimes_{A}Q$. In terms of the components of
the curvature, $\mathcal{J}^{\Delta}$ is given by formulas
\[
\mathcal{J}_{k}^{\Delta}(\xi_{1}\cdots\xi_{\ell}\otimes q):=\sum_{\sigma\in
S_{\ell-k+1,k-1}}\alpha(\sigma,\boldsymbol{\xi})\xi_{\sigma(1)}\cdots
\xi_{\sigma(\ell-k+1)}\otimes J(\Delta)(\xi_{\sigma(\ell-k+2)},\ldots
,\xi_{\sigma(\ell)}|q),
\]
$\xi_{1},\ldots,\xi_{\ell}\in L$, and $q\in Q$.

\begin{example}
[Example \ref{ExPLR}, Part II]\label{ExPRC} As in Example \ref{ExPLR}, let
$\mathscr{P}$ be a $P_{\infty}$ algebra and $(\mathscr{P},\Omega
^{1}(\mathscr{P}))$ the associated $LR_{\infty}[1]$ algebra. It is easy to see
that there is a unique right $(\mathscr{P},\Omega^{1}(\mathscr{P}))$-module
structure $\Delta$ on $\mathscr{P}$ such that
\begin{equation}
\Delta(df_{1},\ldots,df_{k-1}|f_{k})=-\sigma_{\mathbb{X}}(df_{1}%
,\ldots,df_{k-1}|f_{k}), \label{34}%
\end{equation}
$f_{1},\ldots,f_{k}\in\mathscr{P}$. As an immediate consequence, there is an
additional differential $D^{\Delta}$ on K\"{a}hler forms $\Omega
(\mathscr{P})=S_{\mathscr{P}}^{\bullet}\Omega^{1}(\mathscr{P})$. If
$\mathscr{P}$ is the algebra of smooth functions on a \emph{graded manifold
}$\mathcal{M}$, then formula (\ref{34}) also defines a right
$(\mathscr{P},\Omega^{1}(\mathcal{M}))$-module structure on $\mathscr{P}$.
Therefore, in this case, there is an additional differential on differential
forms $\Omega(\mathcal{M})$.
\end{example}

\begin{example}
[Example \ref{ExJac}, Part II]\label{ExJacII} As in Example \ref{ExJac}, let
$\mathscr{C}$ be a $J_{\infty}$ algebra and $(\mathscr{C},J^{1}\mathscr{C})$
the associated $LR_{\infty}[1]$ algebra. It is easy to see that there is a
unique right $(\mathscr{C},J^{1}\mathscr{C})$-module structure $\Delta$ on
$\mathscr{C}$ such that
\begin{equation}
\Delta(j^{1}f_{1},\ldots,j^{1}f_{k-1}|f_{k})=-\sigma_{\mathbb{X}}(j^{1}%
f_{1},\ldots,j^{1}f_{k-1}|f_{k}),
\end{equation}
$f_{1},\ldots,f_{k}\in\mathscr{P}$. As an immediate consequence, there is a
differential $D^{\Delta}$ on $S_{\mathscr{C}}^{\bullet}J^{1}\mathscr{C}$.
\end{example}

\subsection{Higher Right Schouten-Nijenhuis Calculus}

A right linear connection in a vector bundle, i.e., a right connection along
the LR algebra of vector fields, determines a \textquotedblleft right
version\textquotedblright\ of the standard Schouten-Nijenhuis calculus. Here,
I present a homotopy generalization of it. {This allows me to provide a
homotopy version of the main result of \cite{h98} about the relation between
LR algebras and Batalin-Vilkovisky algebras. Namely, in Section \ref{SecBV} I
show that, under the existence of a suitable right connection, the $P_{\infty
}[1]$ algebra determined by a SH LR algebra is equipped with a structure of
\textquotedblleft homotopy Batalin-Vilkovisky algebra\textquotedblright\ (see
below).}

Let $(A,L)$ be an $LR_{\infty}[1]$ algebra with structure multiderivation
$\mathbb{X}$. In this section I adopt the same notation as in Section
\ref{SecLSN} about the $P_{\infty}[1]$ algebra structure on $S_{A}^{\bullet}%
L$. Consider an $A$-module $Q$ with a right $(A,L)$-connection $\Delta$, and
the corresponding $\mathrm{Sym}_{A}(L,A)$-module derivation $D^{\Delta}%
=D_{1}^{\Delta}+D_{2}^{\Delta}+\cdots$ of $S_{A}^{\bullet}L\otimes_{A}Q$.

\begin{proposition}
\label{Prop3}{for all $u_{1},\ldots,u_{k+1}\in S_{A}^{\bullet}L$,
\begin{equation}
\lbrack\cdots\lbrack\lbrack D_{k}^{\Delta},\mu_{u_{1}}],\mu_{u_{2}}]\cdots
,\mu_{u_{k}}]=\mu_{\{u_{1},\ldots,u_{k}\}}. \label{6}%
\end{equation}
In particular, the operator $D_{k}^{\Delta}$ is a differential operator of
order $k$ in the $S_{A}^{\bullet}L$-module $S_{A}^{\bullet}L\otimes_{A}Q$.}
\end{proposition}

\begin{proof}
{First of all, notice that, since $[\mu_{u},\mu_{v}]=0$ for all tensors $u,v$,
then (\ref{6}) implies that $D_{k}^{\Delta}$ is a differential operator of
order $k$, as claimed. Now I prove (\ref{6}).} Suppose, preliminarly, that
identity (\ref{6}) is true when $u_{1},\ldots,u_{k}$ are either from $A$ or
from $L$. Now, let $u_{i}\in S_{A}^{\ell_{i}}L$, $i=1,\ldots,k$. Since
$[\mathcal{D},\mu_{uv}]=[\mathcal{D},\mu_{u}]\mu_{u}+(-)^{\mathcal{D}u}\mu
_{u}[\mathcal{D},\mu_{v}]$, for all symmetric tensors $u,v$ and any graded
$K$-linear endomorphism $\mathcal{D}$ of $Q\otimes_{A}S_{A}^{\bullet}L$, the
general assertion (\ref{6}) follows from an easy induction on the $\ell_{i}$.
It remains to prove (\ref{6}) when $u_{1},\ldots,u_{k}$ are either $0$ or $1$-tensors.

Put $\mathcal{M}(u_{1},\ldots,u_{k}):=[\cdots\lbrack\lbrack D_{k}^{\Delta}%
,\mu_{u_{1}}],\mu_{u_{2}}]\cdots,\mu_{u_{k}}]$. Since $[\mu_{u},\mu_{v}]=0$
for all $u,v\in S_{A}^{\bullet}L$, then, in view of the Jacobi identity for
the graded commutator, $\mathcal{M}$ is graded symmetric in its arguments.
Moreover, since $D_{k}^{\Delta}$ is a $\mathrm{Sym}_{A}(L,A)$-module
derivation subordinate to $D_k$, $\mathcal{M}$ vanishes whenever two arguments
are from $A$. Indeed, let $a,b\in A$. Then
\[
\lbrack\lbrack D_{k}^{\Delta},\mu_{a}],\mu_{b}]=[[D_{k}^{\Delta},i_{a}%
],i_{b}]=0.
\]
Now, let $a\in A$, and let $\xi_{1},\ldots,\xi_{k-1}\in L$. Using $[\mu_{\xi
},i_{\omega}]=-i_{i_{\xi}\omega}$ for all forms $\omega$ it is easy to see
that
\[
\mathcal{M}(a,\xi_{1},\ldots,\xi_{k-1}){}=\mu_{\{a,\xi_{1},\ldots,\xi_{k-1}%
\}}.
\]
Finally, I have to show that
\[
\mathcal{M}(\xi_{1},\ldots,\xi_{k}){}=\mu_{\{\xi_{1},\ldots,\xi_{k}\}}%
,\quad\xi_{1},\ldots,\xi_{k}\in L.
\]
This follows from analogous straightforward computations as those in the proof
of Proposition \ref{Prop2}.
\end{proof}

For $u_{1},\ldots,u_{k-1}\in S_{A}^{\bullet}L$, and $U\in S_{A}^{\bullet
}L\otimes_{A}Q$, put
\[
R^{\Delta}(u_{1},\ldots,u_{k-1}|U):=[\cdots\lbrack\lbrack D_{k}^{\Delta}%
,\mu_{u_{1}}],\mu_{u_{2}}]\cdots,\mu_{u_{k-1}}]U.
\]

The following theorem can be proved exactly as Theorem \ref{Theor1}.

\begin{theorem}
The sum $R^{\Delta}:=R_{1}^{\Delta}+R_{2}^{\Delta}+\cdots$ is a right
$(K,S_{A}^{\bullet}L)$-connection in $S_{A}^{\bullet}L\otimes_{A}Q$ whose
curvature $J(R^{\Delta})$ is given by%
\begin{equation}
J(R^{\Delta})(u_{1},\ldots,u_{k-1}|U)=[\cdots\lbrack\lbrack\mathcal{J}%
_{k}^{\Delta},\mu_{u_{1}}],\mu_{u_{2}}]\cdots,\mu_{u_{k-1}}]U. \label{25}%
\end{equation}
Moreover,
\begin{equation}
R^{\Delta}(u_{1},\ldots,u_{k-1}|\mu_{u}U)=\mu_{\{u_{1},\ldots,u_{k-1}%
,u\}}U+(-)^{\chi}\mu_{u}R^{\Delta}(u_{1},\ldots,u_{k-1}|U), \label{26}%
\end{equation}
where $\chi=\bar{u}(\bar{u}_{1}+\cdots+\bar{u}_{k-1}+1)$, and
\[
R^{\Delta}(uu_{1},u_{2},\ldots,u_{k-1}|U)=(-)^{u}\mu_{u}R^{\Delta}%
(u_{1},\ldots,u_{k-1}|U)+(-)^{\chi^{\prime}}R^{\Delta}(u,u_{2}\ldots
,u_{k-1}|\mu_{u_{1}}U)
\]

where $\chi^{\prime}=\bar{u}_{1}(\bar{u}_{2}+\cdots+\bar{u}_{k-1})$, for all
tensors $u,u_{1},\ldots,u_{k}$, and all $U\in S_{A}^{\bullet}L\otimes_{A}Q$.
\end{theorem}

\begin{corollary}
Let $\Delta$ be a right $(A,L)$-connection in $Q$. Then $\Delta$ is flat iff
$R^{\Delta}$ equips $S_{A}^{\bullet}L\otimes_{A}Q$ with the structure of a
right $L_{\infty}[1]$ module over the $L_{\infty}[1]$ algebra $S_{A}^{\bullet
}L$.
\end{corollary}

\begin{remark}
Beware that, although right Schouten-Nijenhuis calculus looks formally similar
to left one, the two are actually different. The main difference is that they
live on different spaces: a space of forms for left calculus, and a space of
tensors for right calculus. This is why, for instance, Proposition
\ref{PropBV} in the next section does not have a \textquotedblleft left
calculus\textquotedblright\ analogue.
\end{remark}

\subsection{Right $(A,L)$-Module Structures on $A$\label{SecBV}}

Let $(A,L)$ be an ordinary Lie-Rinehart algebra. Then the exterior algebra
$\Lambda_{A}^{\bullet}L$ of $L$ is equipped with a Gerstenhaber algebra
structure. In \cite{h98} Huebschmann showed that a right $(A,L)$-module
structure $\Delta$ on $A$ determines the structure of a
\emph{Batalin-Vilkovisky }(\emph{BV})\emph{ algebra} on $\Lambda_{A}^{\bullet
}L$. Recall that a BV algebra is an associative, graded commutative, unital
algebra $B$ equipped with a degree $1$, second order differential operator
$\square:B\longrightarrow B$ such that $\square1=0$, and $\square^{2}=0$. Any BV algebra is naturally equipped with a Gerstenhaber bracket determined by $\square$ via the \emph{derived bracket
formula}
\[
\lbrack b_{1},b_{2}]:=[\square,b_{1}],b_{2}]1,\quad b_{1},b_{2}\in B.
\]
In \cite{h98} Huebschmann showed that the Rinehart operator $D^{\Delta}$
associated to a right $(A,L)$-module structure $\Delta$ on $A$ generates the
Gerstenhaber bracket $[{}\cdot\emph{{}},{}\cdot{}]$ of $\Lambda_{A}^{\bullet
}L$ in the sense that $(\Lambda_{A}^{\bullet}L,D^{\Delta},[{}\cdot\emph{{}}%
,{}\cdot{}])$ is a BV algebra. Huebschmann's result has a homotopy analogue.
To show this, I first recall a homotopy analogue of a BV algebra. Let
$\mathscr{B}$ be an associative, graded commutative, unital algebra and
$\square:\mathscr{B}\longrightarrow\mathscr{B}$ any degree $1$, $K$-linear
map. Define operations in $\mathscr{B}$ via the \emph{higher derived bracket
formulas }\cite{k00,v05}%
\begin{equation*}
\Lambda^\prime_{k}(u_{1},u_{2},\ldots,u_{k})=[\cdots\lbrack\lbrack\square
,u_{1}],u_{2}]\cdots,u_{k}]1,\quad k\in\mathbb{N}.%
\end{equation*}
The $\Lambda^\prime_{k}$'s are graded symmetric. In \cite{bda96} it was proved for the
first time that, when $\square1=0$ and $\square^{2}=0$, the $\Lambda^\prime_{k}$
equips $\mathscr{B}$ with the structure of an $L_{\infty}[1]$ algebra.
Now, suppose that $\square = \square_1 + \square_2 + \cdots$, with $\square_k$ a degree $1$, differential operator of order $k$, and, moreover, 
\begin{equation}\label{Eq_BV}
\square_i 1 = 0, \quad \text{and} \quad \sum_{i+j = k}[\square_i, \square_j] = 0
\end{equation}
(so that $\square 1= \square^2 = 0$ as above), and put
\begin{equation}
\Lambda_{k}(u_{1},u_{2},\ldots,u_{k})=[\cdots\lbrack\lbrack\square_k
,u_{1}],u_{2}]\cdots,u_{k}]1,\quad k\in\mathbb{N}. \label{9}
\end{equation}

The $\Lambda_k$'s are different from the $\Lambda^\prime_k$'s in general. However, they also equip $\mathscr{B}$ with the structure of an $L_\infty [1]$-algebra. Moreover, they are multiderivations. So, if one puts $\Lambda=\Lambda_{1}+\Lambda_{2}+\cdots$, then $(\mathscr{B},\Lambda)$ is a $P_{\infty}[1]$
algebra. If all the $\square_i$'s vanish but the second one,
then $(\mathscr{B},\square_2)$ is a BV algebra. In the general case, one can give the following definition \cite{bl13} (see also \cite{k00,bv14}).

\begin{definition}
A $BV_{\infty}$\emph{ algebra} is an associative, graded commutative, unital
algebra $\mathscr{B}$ equipped with a family of degree $1$ differential operators $\square_k : \mathscr{B} \longrightarrow \mathscr{B}$ of order $k = 1, 2, \ldots$ such that (\ref{Eq_BV}) hold.
\end{definition}


\begin{proposition}
\label{PropBV}Let $(A,L)$ be an $LR_\infty [1]$-algebra and $\Delta$ a flat, right $(A,L)$-connection in $A$. Then the operators $D_k^\Delta$ build up a $BV_\infty$ algebra structure generating the $P_\infty [1]$ brackets in $S^\bullet_A L$, i.e.
\begin{equation}\label{brackets}
\{u_{1},\ldots,u_{k}\}=[\cdots\lbrack\lbrack D_{k}^{\Delta},\mu_{u_{1}}%
],\mu_{u_{2}}]\cdots,\mu_{u_{k}}]1.
\end{equation}
for all $u_1,\ldots, u_k \in S^\bullet_A L$.
\end{proposition}
\begin{proof}
First recall that the $D_k^\Delta$'s are $k$-th order differential operators. Properties (\ref{Eq_BV}) follow from $D^\Delta 1 = (D^\Delta)^2 = 0$. Finally, one immediately gets (\ref{brackets}) by applying identity (\ref{6}) to $1 \in S_{A}^{\bullet}L$.
\end{proof}

\begin{example}
Let $V$ be an $L_{\infty}$ algebra. Then $(A,L):=(K,V[1])$ is an $LR_{\infty
}[1]$ algebra. Braun and Lazarev (see \cite{bl13} Example 3.12) have recently remarked that there
is a canonical $BV_{\infty}$ algebra structure on $S_{A}^{\bullet}%
L=S^{\bullet}V[1]$ (see also \cite{bv14} where functoriality issues are discussed in details). It is easy to see that the latter coincides with the
$BV_{\infty}$ algebra determined by the trivial right $(A,L)$-module structure
on $A=K$.
\end{example}

\begin{example}
[Example \ref{ExPLR} Part III]Differential forms on a Poisson manifold is
naturally equipped with a $BV$ algebra structure (see, for instance,
\cite{x99}). There is a homotopy analogue of this fact. Namely, as in Example
\ref{ExPLR}, let $\mathscr{P}$ be a $P_{\infty}$ algebra and
$(\mathscr{P},\Omega^{1}(\mathscr{P}))$ the associated $LR_{\infty}[1]$
algebra. Since $\mathscr{P}$ is a right $(\mathscr{P},\Omega^{1}%
(\mathscr{P}))$-module (see Example \ref{ExPRC}), then $\Omega(\mathscr{P})$
is a $BV_{\infty}$ algebra. Similarly, if $\mathscr{P}$ is the algebra of
smooth functions on a \emph{graded manifold }$\mathcal{M}$, then
$\Omega(\mathcal{M})$ is a $BV_{\infty}$ algebra. When $\mathcal{M}$ is
non-graded, this $BV_{\infty}$ algebra coincides with that in Definition 4.2
of \cite{bl13}.
\end{example}

\begin{example}
[Example \ref{ExJac} Part III]The skew-symmetric algebra of $1$-jets of
functions on a Jacobi manifold is naturally equipped with a $BV$ algebra
structure (see, for instance, \cite{v00}). There is a homotopy (and purely
algebraic) analogue of this fact. Namely, as in Example \ref{ExJac}, let
$\mathscr{C}$ be a $J_{\infty}$ algebra and $(\mathscr{C},J^{1}\mathscr{C})$
the associated $LR_{\infty}[1]$ algebra. Since $\mathscr{C}$ is a right
$(\mathscr{C},J^{1}\mathscr{C})$-module (see Example \ref{ExJacII}), then
$S_{\mathscr{C}}^{\bullet}J^{1}\mathscr{C}$ is a $BV_{\infty}$ algebra.
\end{example}

\section{Derivative Representations up to Homotopy\label{SecDerRep}}

As recalled in Section \ref{SecAct} a Lie algebra may act on a manifold. More
generally, a Lie algebra $\mathfrak{g}$ may act on a vector bundle
$E\longrightarrow M$ by infinitesimal vector bundle automorphisms (not
necessarily base-preserving). Following Kosmann-Schwarzbach \cite{k-s76}, I
call such an action a \emph{derivative representation of} $\mathfrak{g}$. A
derivative representation of $\mathfrak{g}$ determines an action of
$\mathfrak{g}$ on $M$. Kosmann-Schwarzabach and Mackenzie showed that
derivative representations of $\mathfrak{g}$ inducing a given action on the
base are equivalent to representations of the transformation Lie algebroid
$\mathfrak{g}\ltimes M$ \cite{k-sm02}. They also define \emph{derivative
representations of Lie algebroids} and prove an analogous result in this
general context. In this section, I show how to generalize the notion of
derivative representation of a Lie algebroid \cite{k-sm02} to the general
context of SH LR algebras and prove a general version of the
Kosmann-Schwarzbach and Mackenzie theorem. Moreover, I define \emph{right
derivative representations }and extend the result to them.

Let $(A,L)$ be an $LR_{\infty}[1]$ algebra, and $\mathcal{A}$ be an
associative, graded commutative, unital $A$-algebra with a pre-action $\nabla$
of $(A,L)$. Moreover, let $\mathcal{P}$ be a graded, left $\mathcal{A}$-module.

\begin{definition}
\label{DefLDR}A \emph{left derivative pre-representation of }$(A,L)$ \emph{on
}$\mathcal{P}$ \emph{subordinate to }$\nabla$ is a left $(A,L)$-connection
$\nabla^{\mathcal{P}}$ in $\mathcal{P}$ such that, for all $\xi_{1},\ldots
,\xi_{k-1}\in L$, and $k\in\mathbb{N}$, the pair
\[
(\nabla^{\mathcal{P}}(\xi_{1},\ldots,\xi_{k-1}),\nabla(\xi_{1},\ldots
,\xi_{k-1}))
\]
is an $\mathcal{A}$-module derivation of $\mathcal{P}$, i.e.,
\[
\nabla^{\mathcal{P}}(\xi_{1},\ldots,\xi_{k-1}|{}fp)=(-)^{\chi}f\nabla
^{\mathcal{P}}(\xi_{1},\ldots,\xi_{k-1}|{}p)+\nabla(\xi_{1},\ldots,\xi
_{k-1}|{}f{})\,p\,,
\]
where $\chi=(\bar{\xi}_{1}+\cdots+\bar{\xi}_{k-1}+1)\bar{f}$, for all
$f\in\mathcal{A}$, and $p\in\mathcal{P}$. If $\nabla$ is an action, then
$\nabla^{\mathcal{P}}$ is a \emph{left derivative representation} if it is flat.
\end{definition}

With a straightforward computation, one can check the following

\begin{proposition}
\label{Prop4}Let $\nabla^{\mathcal{P}}$ be a left $(A,L)$-connection in
$\mathcal{P}$. Then $\nabla^{\mathcal{P}}$ is a left derivative
pre-representation subordinate to $\nabla$ iff $(D^{\nabla^{\mathcal{P}}%
},D^{\nabla})$ is a left $\mathrm{Sym}_{A}(L,\mathcal{A})$-module derivation
of $\mathrm{Sym}_{A}(L,\mathcal{P})$.
\end{proposition}

The following proposition generalizes the main result of \cite{k-sm02}.

\begin{proposition}
Let $\nabla$ be an action. Left derivative pre-representations of $(A,L)$ on
$\mathcal{P}$ subordinate to $\nabla$ are equivalent to left $(\mathcal{A}%
,\mathcal{A}\otimes_{A}L)$-connections in $\mathcal{P}$.
\end{proposition}

\begin{proof}
Let $\mathcal{P}$ possess a left derivative pre-representation of $(A,L)$
subordinate to $\nabla$. Extending the structure maps by $\mathcal{A}%
$-linearity, one gets a left $(\mathcal{A},\mathcal{A}\otimes_{A}%
L)$-connection. Conversely, let $\mathcal{P}$ possess a left $(\mathcal{A}%
,\mathcal{A}\otimes_{A}L)$-connection. Restricting the structure maps
\[
(\mathcal{A}\otimes_{A}L)\times\cdots\times(\mathcal{A}\otimes_{A}%
L)\times\mathcal{P}\longrightarrow\mathcal{P}%
\]
to $L\times\cdots\times L\times\mathcal{P}$ one gets a left derivative pre-representation.
\end{proof}

Let $\nabla$ be an action and $\nabla^{\mathcal{P}}$ be a left derivative
representation of $(A,L)$. Then $\mathcal{P}$ possesses both a flat left
$(A,L)$-connection $\nabla^{\mathcal{P}}$ and a flat left $(\mathcal{A}%
,\mathcal{A}\otimes_{A}L)$-connection $\bar{\nabla}^{\mathcal{P}}$.
Accordingly, in view of Proposition \ref{Prop4}, two (a-priori different)
Chevalley-Eilenberg DG modules can be associated to it. Namely, $(\mathrm{Sym}%
_{A}(L,\mathcal{P}),D^{\nabla^{\mathcal{P}}})$ and $(\mathrm{Sym}%
_{\mathcal{A}}(\mathcal{A}\otimes_{A}L,\mathcal{P}),D^{\bar{\nabla
}^{\mathcal{P}}})$. It is easy two see that the two do actually identify under
the canonical isomorphism
\[
\mathrm{Sym}_{\mathcal{A}}(\mathcal{A}\otimes_{A}L,\mathcal{P})\simeq
\mathrm{Sym}_{A}(L,\mathcal{P}).
\]

Now, let $\mathcal{Q}$ be another graded $\mathcal{A}$-module.

\begin{definition}
\label{DefRDR}A \emph{right derivative pre-representation of }$(A,L)$ \emph{on
}$\mathcal{Q}$ \emph{subordinate to }$\nabla$ is a right $(A,L)$-connection
$\Delta^{\mathcal{Q}}$ in $\mathcal{Q}$ such that, for all $\xi_{1},\ldots
,\xi_{k-1}\in L$, and $k\in\mathbb{N}$, the pair
\[
(\Delta^{\mathcal{Q}}(\xi_{1},\ldots,\xi_{k-1}{}{}),-\nabla(\xi_{1},\ldots
,\xi_{k-1}))
\]
is an $\mathcal{A}$-module derivation of $\mathcal{Q}$, i.e.,
\[
\Delta^{\mathcal{Q}}(\xi_{1},\ldots,\xi_{k-1}|{}fp)=(-)^{\chi}f\Delta
^{\mathcal{P}}(\xi_{1},\ldots,\xi_{k-1}|{}p)-\nabla(\xi_{1},\ldots,\xi
_{k-1}|{}f{})\,p\,,
\]
for all $f\in\mathcal{A}$, and $q\in\mathcal{Q}$. If $\nabla$ is an action,
then $\Delta^{\mathcal{Q}}$ is a \emph{right derivative representation} if it
is flat.
\end{definition}

With a straightforward computation, one can check the following

\begin{proposition}
\label{Prop5}Let $\Delta^{\mathcal{Q}}$ be a right $(A,L)$-connection in
$\mathcal{Q}$. Then $\Delta^{\mathcal{Q}}$ is a right derivative
pre-representation subordinate to $\nabla$ iff $(D^{\Delta^{\mathcal{Q}}%
},D^{\nabla})$ is a $\mathrm{Sym}_{A}(L,\mathcal{A})$-module derivation of
$\mathcal{Q}\otimes_{A}S_{A}^{\bullet}L$.
\end{proposition}

\begin{proposition}
Let $\nabla$ be an action. Right derivative pre-representations of $(A,L)$ on
$\mathcal{Q}$ subordinate to $\nabla$ are equivalent to right $(\mathcal{A}%
,\mathcal{A}\otimes_{A}L)$-connections in $\mathcal{Q}$.
\end{proposition}

\begin{proof}
Let $\mathcal{Q}$ possess a right derivative pre-representation of $(A,L)$
subordinate to $\nabla$. Extending the structure maps by $\mathcal{A}%
$-linearity, one gets a right $(\mathcal{A},\mathcal{A}\otimes_{A}%
L)$-connection. Conversely, let $\mathcal{Q}$ possess a right $(\mathcal{A}%
,\mathcal{A}\otimes_{A}L)$-connection. Restricting the structure maps
\[
(\mathcal{A}\otimes_{A}L)\times\cdots\times(\mathcal{A}\otimes_{A}%
L)\times\mathcal{Q}\longrightarrow\mathcal{Q}%
\]
to $L\times\cdots\times L\times\mathcal{Q}$ one gets a right derivative pre-representation.
\end{proof}

Finally, Let $\nabla$ be an action and $\Delta^{\mathcal{Q}}$ be a right
derivative representation of $(A,L)$. Then $\mathcal{Q}$ possesses both a flat
right $(A,L)$-connection $\Delta^{\mathcal{Q}}$ and a flat right
$(\mathcal{A},\mathcal{A}\otimes_{A}L)$-connection $\bar{\Delta}^{\mathcal{Q}%
}$. Accordingly, in view of Proposition \ref{Prop5}, two (a-priori different)
DG modules can be associated to it along the lines of Section
\ref{SecRightConn}. Namely, $(S_{A}^{\bullet}L\otimes_{A}\mathcal{Q}%
,D^{\Delta^{\mathcal{Q}}})$ and $(S_{\mathcal{A}}^{\bullet}(\mathcal{A}%
\otimes_{A}L)\otimes_{\mathcal{A}}\mathcal{Q},D^{\bar{\Delta}^{\mathcal{Q}}}%
)$. It is easy two see that the two do actually identify under the canonical
isomorphism
\[
S_{\mathcal{A}}^{\bullet}(\mathcal{A}\otimes_{A}L)\otimes_{\mathcal{A}%
}\mathcal{Q}\simeq S_{A}^{\bullet}L\otimes_{A}\mathcal{Q}.
\]

\subsection*{Acknowledgments}

{I'm much indebted to Jim Stasheff for carefully reading all preliminary
versions of this paper. This presentation has been greatly influenced by his
numerous comments and suggestions. I also thank Marco Zambon for his comments
on a preliminary draft of the manuscript. They suggested to me to include in
the paper Appendix \ref{Tables}.}

\appendix{}

\section{Graded Symmetric Forms and Tensors\label{ApForm}}

\label{SecMultAlg}

Let $A$ be a graded, associative, graded commutative, unital algebra over a
field $K$ of zero characteristic, and $L$ a graded $A$-module.

Consider the space $\mathrm{Sym}_{A}^{k}(L,A)$
of $A$-multilinear, graded symmetric maps $L^{\times k}\longrightarrow A$.
Elements in $\mathrm{Sym}_{A}^{k}(L,A)$ will be called \emph{symmetric }%
$k$\emph{-forms }(on $L$), or simply $k$\emph{-forms}. The direct sum
$\mathrm{Sym}_{A}(L,A):=\bigoplus_{k\geq0}\mathrm{Sym}_{A}^{k}(L,A)$ is a
bi-graded algebra, and an associative, graded commutative, unital algebra with
product given by
\begin{equation}
\omega\omega^{\prime}(\xi_{1},\ldots,\xi_{k+k^{\prime}})=\sum_{\sigma\in
S_{k,k^{\prime}}}\alpha(\sigma,\boldsymbol{\xi})(-)^{\chi}\omega(\xi
_{\sigma(1)},\ldots,\xi_{\sigma(k)})\omega^{\prime}(\xi_{\sigma(k+1)}%
,\ldots,\xi_{\sigma(k+k^{\prime})}), \label{1}%
\end{equation}
$\chi=\bar{\omega}^{\prime}(\bar{\xi}_{\sigma(1)}+\cdots+\bar{\xi}_{\sigma
(k)})$, $\omega$ a $k$-form, $\omega^{\prime}$ a $k^{\prime}$-form, and
$\xi_{1},\ldots,\xi_{k+k^{\prime}}\in L$.

Consider also the $k$-th symmetric power $S_{A}^{k}L$ of $L$. Elements in
$S_{A}^{k}L$ will be called \emph{symmetric }$k$\emph{-tensors }(on $L$), or
simply $k$\emph{-tensors}. The symmetric algebra $S_{A}^{\bullet}%
L=\bigoplus_{k\geq0}S_{A}^{k}L$ is a bi-graded algebra, and an associative,
graded commutative, unital algebra with product given by the graded symmetric
(tensor) product.

Let $P$ be another $A$-module. Consider the space $\mathrm{Sym}_{A}^{k}(L,P)$
of graded, graded $A$-multilinear, graded symmetric maps $L^{\times
k}\longrightarrow P$. Elements in $\mathrm{Sym}_{A}^{k}(L,P)$ will be called
\emph{symmetric }$P$\emph{-valued }$k$\emph{-forms}, or simply $P$%
\emph{-valued }$k$\emph{-forms}. The direct sum $\mathrm{Sym}_{A}%
(L,P):=\bigoplus_{k\geq0}\mathrm{Sym}_{A}^{k}(L,P)$ is a bi-graded $A$-module,
and a left $\mathrm{Sym}_{A}(L,A)$-module with structure map written
$(\omega,\Omega)\longmapsto\mu_{\omega}\Omega$, and given by the same formula
as (\ref{1}). The space $\mathrm{Sym}_{A}(L,P)$ is also a $S_{A}^{\bullet}%
L$-module with structure map written $(u,\Omega)\longmapsto i_{u}\Omega$, and
given by
\begin{equation}
(i_{u}\Omega)(\xi_{1},\ldots,\xi_{\ell-k}):=(-)^{u\Omega}\Omega(\zeta
_{1},\ldots,\zeta_{k},\xi_{1},\ldots,\xi_{\ell-k}), \label{10}%
\end{equation}
$u=\zeta_{1}\cdots\zeta_{k}$ a $k$-tensor, $\zeta_{1},\ldots,\zeta_{k}\in L$,
and $\Omega$ an $P$-valued $k$-form.

Now, let $Q$ be a third $A$-module. The tensor product $Q\otimes_{A}%
S_{A}^{\bullet}L$ is a $\mathrm{Sym}_{A}(L,A)$-module with structure map
$(\omega,U)\longmapsto i_{\omega}U$ given by
\begin{equation}
i_{\omega}U:=\sum_{\sigma\in S_{k,\ell-k}}\alpha(\sigma;\boldsymbol{\xi
})(-)^{\chi^{\prime}}\xi_{\sigma(1)}\cdots\xi_{\sigma(k)}\otimes\omega
(\xi_{\sigma(k+1)},\ldots,\xi_{\sigma(\ell)})q, \label{11}%
\end{equation}
$\chi^{\prime}=\bar{\omega}(\bar{\xi}_{\sigma(1)}+\cdots+\bar{\xi}_{\sigma
(k)})$, $\omega$ a $k$-form, $U=q\otimes\xi_{1}\cdots\xi_{\ell}\in
Q\otimes_{A}S_{A}^{\ell}L$, $q\in Q$, $\xi_{1},\ldots,\xi_{\ell}\in L$. The
space $S_{A}^{\bullet}L\otimes_{A}Q$ is also a $S_{A}^{\bullet}L$-module with
obvious structure map written $(u,U)\longmapsto\mu_{u}U$.

\begin{remark}
Let $u$ be a $k$-tensor, and let $\omega$ be a $k$-form. If we understand
$\mathrm{Sym}_{A}(L,P)$ as a module over $\mathrm{Sym}_{A}(L,A)$ (resp.,
$S_{A}^{\bullet}L$), then $i_{u}$ (resp., $\mu_{\omega}$) is a differential
operator of order $k$. Similarly, if we understand $Q\otimes_{A}S_{A}%
^{\bullet}L$ as a module over $\mathrm{Sym}_{A}(L,A)$ (resp., $S_{A}^{\bullet
}L$), then $\mu_{u}$ (resp., $i_{\omega}$) is a differential operator of order
$k$.
\end{remark}

{Now, let $\mathcal{A}$ be an associative, graded commutative, unital
$A$-algebra, and let $\mathcal{P}$ and $\mathcal{Q}$ be $\mathcal{A}$-modules.
In particular, they are }$A$-modules. {Clearly, $S_{A}^{\bullet}L\otimes
_{A}\mathcal{A}$ is a $\mathrm{Sym}_{A}(L,A)$-algebra. Moreover, Formula
(\ref{1}) defines a $\mathrm{Sym}_{A}(L,A)$-algebra structure on
$\mathrm{Sym}_{A}(L,\mathcal{A})$. A similar formula defines a $\mathrm{Sym}%
_{A}(L,\mathcal{A})$-module structure on $\mathrm{Sym}_{A}(L,\mathcal{P})$.
Finally, Formula (\ref{11}) defines a $\mathrm{Sym}_{A}(L,\mathcal{A})$-module
structure on ${S_{A}^{\bullet}}${$L$}$\otimes_{A}\mathcal{Q}$.}

\section{Operations with Right Connections\label{ApOpRight}}

Right $(A,L)$-connections can be operated with left $(A,L)$-connections as
follows. Let $(P,\nabla)$, and $(P^{\prime},\nabla^{\prime})$ be $A$-modules
with left $(A,L)$-connections, and $(Q,\Delta)$, and $(Q^{\prime}%
,\Delta^{\prime})$ be right $A$-modules with right $(A,L)$-connections. It is
easy to see that formulas
\[
\Delta^{\otimes}(\xi_{1},\ldots,\xi_{k-1}|q\otimes q^{\prime}):=-\Delta
(\xi_{1},\ldots,\xi_{k-1}|q)\otimes q^{\prime}-(-)^{\chi+q}q\otimes
\Delta^{\prime}(\xi_{1},\ldots,\xi_{k-1}|q^{\prime}),
\]
where $\chi=(\bar{\xi}_{1}+\cdots+\bar{\xi}_{k-1})\bar{q}$, define a left
$(A,L)$-connection $\Delta^{\otimes}$ in $Q\otimes_{A}Q^{\prime}$ (beware,
\emph{left} not right). A straightforward computation shows that the curvature
of $\Delta^{\otimes}$ is given by formulas
\[
J(\Delta^{\otimes})(\xi_{1},\ldots,\xi_{k-1}|q\otimes q^{\prime}%
)=-J(\Delta)(\xi_{1},\ldots,\xi_{k-1}|q)\otimes q^{\prime}-(-)^{\chi}q\otimes
J(\Delta^{\prime})(\xi_{1},\ldots,\xi_{k-1}|q^{\prime}).
\]
In particular, if $\Delta$ and $\Delta^{\prime}$ are flat, then $\Delta
^{\otimes}$ is flat as well.

Similarly, formulas
\[
\Delta^{\mathrm{Hom}}(\xi_{1},\ldots,\xi_{k-1}|\varphi)(q):=(-)^{\chi^{\prime
}+\varphi}\varphi(\Delta(\xi_{1},\ldots,\xi_{k-1}|q))-\Delta^{\prime}(\xi
_{1},\ldots,\xi_{k-1}|\varphi(q)),
\]
where $\chi^{\prime}=(\bar{\xi}_{1}+\cdots+\bar{\xi}_{k-1})\bar{\varphi}$,
define a left $(A,L)$-connection $\Delta^{\mathrm{Hom}}$ in $\mathrm{Hom}%
_{A}(Q,Q^{\prime})$. A straightforward computation shows that the curvature of
$\Delta^{\mathrm{Hom}}$ is given by formulas
\[
J(\Delta^{\mathrm{Hom}})(\xi_{1},\ldots,\xi_{k-1}|\varphi)(q):=(-)^{\chi
^{\prime}}\varphi(J(\Delta)(\xi_{1},\ldots,\xi_{k-1}|q))-J(\Delta^{\prime
})(\xi_{1},\ldots,\xi_{k-1}|\varphi(q)).
\]
In particular, if $\Delta$ and $\Delta^{\prime}$ are flat, then $\Delta
^{\mathrm{Hom}}$ is flat as well.

Finally, formulas%
\[
\Diamond(\xi_{1},\ldots,\xi_{k-1}|p\otimes q):=(-)^{\psi+p}p\otimes\Delta
(\xi_{1},\ldots,\xi_{k-1}|q)-\nabla(\xi_{1},\ldots,\xi_{k-1}|p)\otimes q,
\]%
\[
\Diamond^{\prime}(\xi_{1},\ldots,\xi_{k-1}|\varphi)(p):=(-)^{\psi^{\prime
}+\varphi}\varphi(\nabla(\xi_{1},\ldots,\xi_{k-1}|p))+\Delta(\xi_{1}%
,\ldots,\xi_{k-1}|\varphi(p)),
\]
and%
\[
\Diamond^{\prime\prime}(\xi_{1},\ldots,\xi_{k-1}|\varphi)(q):=-\nabla(\xi
_{1},\ldots,\xi_{k-1}|\varphi(q))-(-)^{\psi^{\prime}+\varphi}\varphi
(\Delta(\xi_{1},\ldots,\xi_{k-1}|q)),
\]
where $\psi=(\bar{\xi}_{1}+\cdots+\bar{\xi}_{k-1})\bar{p}$, and $\psi^{\prime
}=(\bar{\xi}_{1}+\cdots+\bar{\xi}_{k-1})\bar{\varphi}$, define right
$(A,L)$-connections $\Diamond$, $\Diamond^{\prime}$, and $\Diamond
^{\prime\prime}$ in $P\otimes_{A}Q$, $\mathrm{Hom}_{A}(P,Q)$, and
$\mathrm{Hom}_{A}(Q,P)$, respectively. The respective curvatures are given by
formulas
\[
J(\Diamond)(\xi_{1},\ldots,\xi_{k-1}|p\otimes q):=(-)^{\psi}p\otimes
J(\Delta)(\xi_{1},\ldots,\xi_{k-1}|q)-J(\nabla)(\xi_{1},\ldots,\xi
_{k-1}|p)\otimes q,
\]%
\[
J(\Diamond^{\prime})(\xi_{1},\ldots,\xi_{k-1}|\varphi)(p):=(-)^{\psi^{\prime}%
}\varphi(J(\nabla)(\xi_{1},\ldots,\xi_{k-1}|p))+J(\Delta)(\xi_{1},\ldots
,\xi_{k-1}|\varphi(p)),
\]
and%
\[
J(\Diamond^{\prime\prime})(\xi_{1},\ldots,\xi_{k-1}|\varphi)(q):=-J(\nabla
)(\xi_{1},\ldots,\xi_{k-1}|\varphi(q))-(-)^{\psi^{\prime}}\varphi
(J(\Delta)(\xi_{1},\ldots,\xi_{k-1}|q))
\]

The straightforward details are left to the reader.

\begin{remark}
Let $Q$ be an $A$-module with a right $(A,L)$-connection $\Delta$. There is an
induced left $(A,L)$-connection $\Delta^{\mathrm{End}}$ in $\mathrm{End}_{A}%
Q$. In its turn, $\Delta^{\mathrm{End}}$ determines a derivation
$D^{\Delta^{\mathrm{End}}}$ of $\mathrm{End}_{A}Q$-valued forms over $L$. On
the other hand, in view of Remark \ref{RemRCurv}, the curvature $J(\Delta)$ of
$\Delta$ is an $\mathrm{End}_{A}Q$-valued form (with infinitely many
components with a definite number of entries) and
\[
D^{\Delta^{\mathrm{End}}}J(\Delta)=\Delta^{\mathrm{End}}\circ J(\Delta
)-J(\Delta)\circ X=-[\Delta,J(\Delta)]-J(\nabla)\circ X=0,
\]
where I used the right Bianchi identity \ref{RBianchi}.
\end{remark}

\section{Tables of Correspondences\label{Tables}}

The constructions described in this paper generalize standard constructions on
Lie algebroids in two directions: on one side in the direction of abstract
algebra (Lie-Rinehart algebras, etc.), on another side in the direction of
higher homotopy theory ($L_{\infty}$ algebras, etc.). To make manifest the
correspondence between standard notions and notions in this paper, I record
below three tables of correspondences. I hope this will help the reader in
orienteering in the zoo of structures discussed in this paper. Notions with
the same label (a Roman number I, II, III, \ldots) do correspond to each
other. I added a bibliographic reference to less standard notions, and I left
an empty space where a notion is empty or, to my knowledge, has not been
discussed in literature. The second column of Table \ref{Table} contains the
general notions discussed in this paper.

\begin{table}[ht]
\caption{\textbf{Lie Theory}}
\begin{tabular}{ccc}
\hline
     & \textbf{Standard}                & \textbf{Higher Homotopy}     \\ \hline
I    & Lie algebra  $\mathfrak{g}$                    & $L_\infty$ algebra $V$ \cite{ls93, lm95}   \\ \hline
II   & ---                              & ---              \\ \hline
III  & left representations             & $L_\infty$ modules    \cite{lm95}          \\ \hline
IV   & right representations            & ---               \\ \hline
V    & CE differential                  & homological vector field on $V[1]$ \cite{cs11}\\ \hline
VI   & ---                              & ---                        \\ \hline
VII  & linear Poisson structure on $\mathfrak g^\ast$    & ---                     \\ \hline
VIII & actions on manifolds             & actions on graded manifolds  \cite{mz12}\\ \hline
IX   & transformation Lie algebroid     & transformation $L_\infty$ algebroid  \cite{mz12} \\ \hline
X    & linear actions on vector bundles \cite{k-s76, k-sm02}& ---                   \\ \hline
XI   & BV structure on the CE complex  \cite{l92,bl13} & $BV_\infty$ structure on $S^\bullet V[1]$   \cite{bl13,bv14}     \\ \hline
\end{tabular}
\end{table}

\begin{table}[ht]
\caption{\textbf{Geometry: Lie Algebroids}}
\begin{tabular}{cccll}
\hline
     & \textbf{Standard}                & \textbf{Higher Homotopy}     \\ \hline
I    & Lie algebroid $E$         & $L_\infty$ algebroid $E$ \cite{sss09, b11, bp12}   \\ \hline
II   & $E$-connections \cite{elw99, x99}                            & ---              \\ \hline
III  & left representations             & ---        \\ \hline
IV   & right representations \cite{p83,mp89}           & ---               \\ \hline
V    & de Rham differential             & CE differential \cite{sss09}\\ \hline
VI   & Spencer operator \cite{p83,kv98}                             & ---                        \\ \hline
VII  & f.w.~linear Poisson structure on $E^\ast$    & sh Poisson structure on $E^\ast$ \cite{b11}                    \\ \hline
VIII & actions on manifolds             & --- \\ \hline
IX   & transformation Lie algebroid \cite{hm90}    & --- \\ \hline
X    & linear actions on vector bundles \cite{k-sm02}& ---                   \\ \hline
XI   & right reps.~and BV structs.  \cite{k85,x99} & ---    \\ \hline
\end{tabular}
\end{table}

\begin{table}[ht]
\caption{\textbf{Algebra: Lie-Rinehart Algebras}}
\begin{tabular}{cccll}
\hline
     & \textbf{Standard}                & \textbf{Higher Homotopy}     \\ \hline
I    & Lie-Rinehart algebra $(A,L)$                 & $LR_\infty [1]$ algebra $(A,L)$ \cite{k01,v12,h13} (Def.~\ref{DefSHLR})  \\ \hline
II   & ---         &$(A,L)$-connections (Defs.~\ref{DefLConn}, \ref{DefRConn})         \\ \hline
III  & left modules             & left modules (Def.~\ref{DefLConn})      \\ \hline
IV   & right modules \cite{p83,mp89}           & right modules (Def.~\ref{DefRConn})              \\ \hline
V    & CE differential             & CE differential \cite{v12,h13} (Eqs.~(\ref{CED}), (\ref{CEDM}))\\ \hline
VI   & Rinehart differential \cite{r63,h98,h99}                             & Rinehart differential (Eq.~(\ref{Rine}))                    \\ \hline
VII  & Gerstenhaber structure on $\Lambda^\bullet_A L$    & $P_\infty$ algebra structure on $S_A^\bullet L$ (Sec.~\ref{SecPoi})                    \\ \hline
VIII & ---          & actions on algebra extensions (Def.~\ref{DefAct})\\ \hline
IX   & ---   & transformation $LR_\infty[1]$ algebra (Cor.~\ref{Cor1}) \\ \hline
X    & --- & derivative representations (Defs.~\ref{DefLDR}, \ref{DefRDR}) \\ \hline
XI   & right modules and BV structs. \cite{h98} & right modules and $BV_\infty$ structs. (Prop.~\ref{PropBV})    \\ \hline
\end{tabular}
\end{table}

\quad


\newpage

\begin{thebibliography}{99}                                                                                               %


\bibitem {ac11}C.~A.~Abad, and M.~Crainic, Representations up to
homotopy of Lie algebroids, \emph{J.~reine angew.~Math.~}\textbf{663}
(2012) 91--126; e-print: arXiv:0901.0319.

\bibitem {bv14}D.~Bashkirov, and A.~A.~Voronov, The BV formalism for $L_\infty$-algebras; e-print: arXiv:1410.6432.

\bibitem {bda96}K.~Bering, P.~H.~Damgaard, and J.~Alfaro, Algebra of
higher antibrackets, \emph{Nuclear Phys.~}\textbf{B478} (1996) 459--503; e-print: arXiv:hep-th/9604027.

\bibitem {bp12}{G.~Bonavolont\`{a}, and N.~Poncin, On the category of Lie
$n$-algebroids, \emph{J.~Geom.~Phys.~}\textbf{73} (2013) 70--90; e-print:
arXiv:1207.3590.}

\bibitem {bl13}C.~Braun, and A.~Lazarev, Homotopy BV algebras in Poisson
geometry, \emph{Trans.~Moscow Math.~Soc.~}\textbf{74} (2013) 217--227; e-print: arXiv:1304.6373.

\bibitem {b11}A.~J.~Bruce, From $L_{\infty}$-algebroids to higher
Schouten/Poisson structures, \emph{Rep.~Math.~Phys.~}\textbf{67 }(2011)
157--177; e-print: arXiv:1007.1389.

\bibitem {cf07}A.~Cattaneo, and G.~Felder, Relative formality theorem and
quantisation of coisotropic submanifolds, \emph{Adv.~Math.~}\textbf{208}
(2007) 521--548; e-print: arXiv:math/0501540.

\bibitem {c00}{M.~Crainic, Chern characters via connections up to homotopy;
e-print: arXiv:math/0009229.}

\bibitem {cm08}M.~Crainic, and I.~Moerdijk, Deformations of Lie brackets:
cohomological aspects, \emph{JEMS} \textbf{10} (2008) 1037--1059; e-print: arXiv:math/0403434.

\bibitem {cs11}{A.~S.~Cattaneo, and F.~Sch\"{a}tz, Introduction to
supergeometry, \emph{Rev.~Math.~Phys.~}\textbf{23} (2011) 669--690;
e-print: arXiv:1011.3401.}

\bibitem {elw99}S.~Evens, J.-H.~Lu, and A.~Weinstein, Transverse
measures, the modular class and a cohomology pairing for Lie algebroids,
\emph{Quart.~J.~Math.~Oxford Ser.~}\textbf{50} (1999) 417--436; e-print:arXiv:dg-ga/9610008.

\bibitem {gtv12}I.~Galvez-Carrillo, A.~Tonks, and B.~Vallette, Homotopy
Batalin-Vilkovisky algebras, \emph{J.~Noncommut.~Geom.~}\textbf{6}
(2012) 539--602; e-print: arXiv:0907.2246.

\bibitem {g63}M.~Gerstenhaber, The cohomology structure of an associative
ring, \emph{Ann.~Math.~}\textbf{78} (1963) 267--288.

\bibitem {hm90}P.~J.~Higgins, and K.~C.~H.~Mackenzie, Algebraic
constructions in the category of Lie algebroids, \emph{J.~Algebra}
\textbf{129} (1990) 194--230.

\bibitem {h90}J.~Huebschmann, Poisson cohomology and quantization,
\emph{J.~reine angew.~Math.~}\textbf{408} (1990) 57--113; e-print: arXiv:1303.3903.

\bibitem {h98}J.~Huebschmann, Lie-Rinehart algebras, Gerstenhaber algebras
and Batalin-Vilkovisky algebras, \emph{Ann.~Inst.~Fourier} \textbf{48}
(1998) 425--440; e-print: arXiv:dg-ga/9704005.

\bibitem {h99}J.~Huebschmann, Duality for Lie-Rinehart algebras and the
modular class, \emph{J.~reine angew.~Math.~}\textbf{510} (1999)
103--159; e-print: arXiv:dg-ga/9702008.

\bibitem {h04}J.~Huebschmann, Lie-Rinehart algebras, Descent and
quantization, in: Galois theory, Hopf algebras, and semiabelian categories
(G.~Janelidze, B.~Pareigis, and W.~Tholen, eds.) \emph{Fields Inst.~Commun.~}\textbf{43} (2004) 295--316; e-print: arXiv:math/0303016.

\bibitem {h05}J.~Huebschmann, Higher homotopies and Maurer-Cartan algebras:
Quasi-Lie-Rinehart, Gerstenhaber, and Batalin-Vilkovisky algebras, in: The
Breadth of symplectic and Poisson geometry, \emph{Progr.~in Math.~}\textbf{232} (2005) 237--302; e-print: arXiv: math/0311294.

\bibitem {h13}{J.~Huebschmann, }Multi derivation Maurer-Cartan algebras and
sh-Lie-Rinehart algebras, e-print: arXiv:1303.4665.

\bibitem {k01}L.~Kjeseth, Homotopy Rinehart cohomology of homotopy
Lie-Rinehart pairs, \emph{Homol.~Homot.~Appl.~}\textbf{3} (2001) 139--163.

\bibitem {k01b}L.~Kjeseth, A Homotopy Lie-Rinehart resolution and classical
BRST cohomology, \emph{Homol.~Homot.~Appl.~}\textbf{3} (2001) 165--192.

\bibitem {k-s76}Y.~Kosmann, On Lie transformation groups and the covariance
of differential operators, in: Differential Geometry and Relativity (M.~Cahen, and M.~Flato, eds.), \emph{Math.~Phys.~Appl.~Math.~}\textbf{3} (1976) 75--89.

\bibitem {k-sm02}Y.~Kosmann-Schwarzbach, and K.~C.~H.~Mackenzie,
Differential operators and actions of Lie algebroids, , in Quantization,
Poisson Brackets and Beyond (T.~Voronov, ed.), Contemp.~Math.~vol.~315, Amer.~Math.~Soc., Providence, RI, 2002, pp. 213--233; e-print: arXiv:math/0209337.

\bibitem {k85}J.-L.~Koszul, Crochet de Schouten-Nijenhuis et cohomologie,
\emph{Ast\'{e}risque}, hors s\'{e}rie (1985) 257--271.

\bibitem {ks07}A.~Kotov, and T.~Strobl, Characteristic classes associated
to $Q$-bundles, e-print: arXiv:0711.4106.

\bibitem {kk12}{N.~Kowalzig, and U.~Kr\"{a}hmer, Batalin-Vilkovsky
structures on $\mathrm{Ext}$ and $\mathrm{Tor}$, \emph{J.~reine angew.~Math.~}(2012) published online: DOI: 10.1515/crelle-2012-0086; e-print:
arXiv:1203.4984.}

\bibitem {k97}I.~S.~Krasil'shchik, Calculus over commutative algebras: a
concise user's guide, \emph{Acta Appl.~Math.~}\textbf{49} (1997) 235---248.

\bibitem {kv98}I.~S.~Krasil'shchik, and A.~M.~Verbovetsky, Homological
methods in equation of mathematical physics, \emph{Advanced Texts in
Mathematics}, Open Education \& Sciences, Opava, 1998, Sections 1.6--1.7;
e-print: arXiv:math/9808130.

\bibitem {k00}O.~Kravchenko, Deformations of Batalin-Vilkovisky algebras,
in: Poisson geometry 1998 (J.~Grabowski, and P.~Urbanski, eds.), vol.~51
of Banach Center Publ., Polish Acad. Sci., Warsaw, 2000, pp.~131--139;
e-print: arXiv:math/9903191.

\bibitem {lm95}T.~Lada, and M.~Markl, Strongly homotopy Lie algebras,
\emph{Comm.~Algebra }\textbf{23 }(1996) 2147--2161; e-print: arXiv:hep-th/9406095.

\bibitem {ls93}T.~Lada, and J.~Stasheff, Introduction to sh Lie algebras
for physicists, \emph{Int.~J.~Theor.~Phys.~}\textbf{32} (1993)
1087--1103; e-print: arXiv:hep-th/9209099.

\bibitem {l92}J.-L.~Loday, Cyclic Homology, Grundlehren der mathematischen
Wissenschaften Vol. 301, Springer-Verlag, Berlin, 1992, Section 10.1.

\bibitem {mp89}Yu.~I.~Manin, and I.~B.~Penkov, The formalism of left
and right connections on supermanifolds, in: Lectures on supermanifolds,
geometrical methods and conformal groups (H.-D.~Doebner, J.~D.~Hennig,
and T.~D.~Palev, eds.) World Scientific Publishing Co.~Inc., Teaneck,
NJ, 1989, pp.~3--13.

\bibitem {m06}R.~A.~Mehta, Supergroupoids, double structures, and
equivariant cohomology, PhD thesis, University of California, Berkeley, 2006;
e-print: arXiv:math.DG/0605356.

\bibitem {mz12}R.~A.~Mehta, and M.~Zambon, $L_{\infty}$-algebra actions,
\emph{Diff.~Geom.~Appl.~}\textbf{30 }(2012) 576--587; e-print: arXiv:1202.2607.

\bibitem {op05}Y.-G.~Oh, and J.-S.~Park, Deformations of coisotropic
submanifolds and strong homotopy Lie algebroids, \emph{Invent. Math.
}\textbf{161} (2005) 287--360; e-print: arXiv:math/0305292.

\bibitem {p83}I.~B.~Penkov, $\mathcal{D}$-modules on supermanifolds,
\emph{Invent.~Math.~}\textbf{71} (1983) 501--512.

\bibitem {r63}G.~Rinehart, Differential forms for general commutative
algebras, \emph{Trans.~Amer.~Math.~Soc.}{} \textbf{108} (1963) 195--222.

\bibitem {sss09}H.~Sati, U.~Schreiber, and J.~Stasheff, Twisted
differential String and Fivebrane structures, \emph{Commun.~Math.~Phys.~}\textbf{315} (2012) 169--213, Appendix A; e-print: arXiv:0910.4001.

\bibitem {s09}F.~Sch\"{a}tz, BFV-complex and higher homotopy structures,
\emph{Commun.~Math.~Phys.~}\textbf{286 }(2009) 399--443; e-print: arXiv:math/0611912.

\bibitem {s94}J.-P.~Schneiders, An introduction to $\mathcal{D}$-modules,
\emph{Bull.~Soc.~Roy.~Sci.~Li\`{e}ge} \textbf{63} (1994) 223--295.

\bibitem {sz11}{Y.~Sheng, C.~Zhu, Higher extensions of Lie algebroids and
applications to Courant algebroids, e-print: arXiv:11035920.}

\bibitem {t74}W.~Tulczyjew, Hamiltonian systems, Lagrangian systems, and the
Legendre transformation, \emph{Symposia Math.} \textbf{14} (1974) 101--114.

\bibitem {u11}K.~Uchino, Derived brackets and sh Leibniz algebras,
\emph{J.~Pure Appl.~Algebra} \textbf{215} (2011) 1102--1111; e-print: arXiv:0902.0044.

\bibitem {v12}L.~Vitagliano, On the strong homotopy Lie-Rinehart algebra of
a foliation, \emph{Commun.~Contemp.~Math.~}\textbf{16} (2014) 1450007 (49 pages); e-print: arXiv:1204.2467.

\bibitem {v12b}L.~Vitagliano, On the strong homotopy associative algebra of
a foliation, \emph{Commun.~Contemp.~Math.~}(2014) 1450026 (34 pages), in press. doi: 10.1142/S0219199714500266; e-print: arXiv:1212.1090.

\bibitem {v00}I.~Vaisman, The BV-algebra of a Jacobi manifold, \emph{Ann.~Polon.~Math.~}\textbf{73} (2000) 275--290; e-print: arXiv:math.DG/9904112.

\bibitem {v99}A.~Voronov, Homotopy Gerstenhaber Algebras, in: Conference
Moshe Flato 1999 (G.~Dito and D.~Sternheimer, eds.), vol.~2 Kluwer
Academic Publishers, the Netherlands, 2000, pp. 307--331; e-print: arXiv:math/9908040.

\bibitem {v05}T.~Voronov, Higher derived brackets and homotopy algebras,
\emph{J.~Pure Appl.~Algebra} \textbf{202} (2005) 133--153; e-print: arXiv:math/0304038.

\bibitem {x99}P.~Xu, Gerstenhaber algebras and BV-algebras in Poisson
geometry, \emph{Comm.~Math.~Phys.~}\textbf{200} (1999) 545--560;
e-print: arXiv:dg-ga/9703001.

\bibitem {y11}S.~Yu, Dolbeault dga of formal neighborhoods and $L_{\infty}%
$-algebroids, poster presented at AGNES 2011, Stony Brook, NY, 2011, url: http://www.math.psu.edu/yu/slides/poster.pdf.
\end{thebibliography}
\end{document}